\documentclass[letterpaper, 11pt,  reqno]{amsart}
\usepackage[margin=1.2in,marginparwidth=1.5cm, marginparsep=0.5cm]{geometry}

\usepackage[hidelinks]{hyperref}
\usepackage{amsmath,amssymb,amscd,amsthm,amsxtra, esint, xcolor,mathtools,enumerate}

\usepackage{bbm}
\usepackage{cancel}
\usepackage{comment}

\allowdisplaybreaks[2]

\usepackage{color}

\newtheorem{theorem}{Theorem} [section]

\newtheorem{lemma}[theorem]{Lemma}
\newtheorem{proposition}[theorem]{Proposition}
\newtheorem{remark}[theorem]{Remark}

\newtheorem{corollary}[theorem]{Corollary}

\numberwithin{equation}{section}
\numberwithin{theorem}{section}

\makeatletter
\@namedef{subjclassname@2020}{\textup{2020} Mathematics Subject Classification}
\makeatother

\begin{document}
	\baselineskip = 14pt

	\title[Temporal variations and estimation for fractional SHEs]
	{Exact temporal variations and parameter estimation for space--time fractional stochastic heat equations driven by multiplicative noise}

	\author[Y. Li, Y. Hu and L. Yan]
	{Yongkang Li, Yaozhong Hu and Litan Yan}
	
	\thanks{Y. Li is supported by the Fundamental Research Funds for the Central Universities of China (CUSF-DH-T-2024066). L. Yan is supported by the National Natural Science Foundation of China (No. 11971101) and the Natural Science Foundation of Shanghai Municipality (24ZR1402900). Y. Hu is supported by   
		NSERC Discovery grant RGPIN
		2024-05941  and a centennial  fund of University of Alberta.}

	\address{Yongkang Li\\
		School of Mathematics and Statistics, Donghua University, 2999 North Renmin Rd., Songjiang, Shanghai 201620, P.R. China
	}
	\address{
		Department of Mathematical and Statistical Sciences, University of Alberta, Edmonton, AB T6G 2G1, Canada}
	\email{yongkangli@mail.dhu.edu.cn}

	\address{Yaozhong Hu\\
		Department of Mathematical and Statistical Sciences, University of Alberta, Edmonton, AB T6G 2G1, Canada}
	\email{yaozhong@ualberta.ca}

	\address{
		Litan Yan\\
		School of Mathematics and Statistics, Donghua University, 2999 North Renmin Rd., Songjiang, Shanghai 201620, P.R. China
	}
	
	\email{litan-yan@hotmail.com}

	\subjclass[2020]{60H15, 60G15, 60H05, 60H30, 60H07}

	\keywords{Temporal power variation; parameter estimation; high-frequency observations; stochastic partial differential equations; fractional stochastic heat equations.
	}
	
	\begin{abstract}
		We establish exact limits for the normalized temporal power variations
		of every fixed real order \(\gamma\geq2\) for the mild solution to a
		one-dimensional space--time fractional stochastic heat equation driven
		by multiplicative space--time white noise.  The results include the exact normalized
		quadratic variation and the critical power variation of order
		\(2\alpha/(\alpha-\beta)\) as special cases. We also obtain the
		corresponding spatially averaged temporal variations. As applications,
		we construct consistent estimators for the drift parameter and for the
		ratio \(\beta/\alpha\).
	\end{abstract}

	%
	\maketitle
	%
	
	%

	\baselineskip = 14pt
	
\section{Introduction}
\label{sec:introduction}

We focus on the following one-dimensional space--time fractional
stochastic heat equation:
\begin{equation}
	\label{eq:main-equation}
	\left\{
	\begin{aligned}
		\partial_t^\beta u_\nu(t,x)
		&=
		-\nu(-\Delta)^{\alpha/2}u_\nu(t,x)
		+
		I_t^{1-\beta}
		\left[
		\sigma(u_\nu(t,x))\dot W(t,x)
		\right],
		\qquad t\geq0,\ x\in\mathbb R,\\
		u_\nu(0,x)&=0,
		\qquad x\in\mathbb R,
	\end{aligned}
	\right.
\end{equation}
where \(\dot W\) denotes space--time white noise, \(\nu>0\),
\(\beta\in(0,1)\), and \(\alpha\in(1,2]\). Here
\(\sigma:\mathbb R\to\mathbb R\) is assumed to be globally Lipschitz
continuous.

In \eqref{eq:main-equation}, \(\partial_t^\beta\) denotes the Caputo
fractional derivative, whose formulation originates in Caputo's work
on linear models of dissipation
\cite{caputo1967linear}. The operator
\(I_t^{1-\beta}\) is the Riemann--Liouville fractional integral, while
\(-(-\Delta)^{\alpha/2}\) is the generator of a symmetric
\(\alpha\)-stable process. Thus, \(\beta\) describes temporal memory,
whereas \(\alpha\) determines the degree of nonlocal spatial diffusion.
Such fractional evolution equations arise naturally in models of
anomalous diffusion; see, for example, \cite{metzler2000random}.

The placement of \(I_t^{1-\beta}\) in front of the stochastic forcing
is dictated by the time-fractional Duhamel principle developed in
\cite{umarov2006fractional,umarov2007generalization,
	umarov2012duhamel}. In particular, this choice allows the mild
formulation to be expressed directly in terms of the fundamental
solution of the corresponding deterministic space--time fractional
heat equation. Accordingly, the mild
solution of \eqref{eq:main-equation} is given by
\begin{equation}
	\label{eq:mild-introduction}
	u_\nu(t,x)
	=
	\int_0^t\int_{\mathbb R}
	G_{\alpha,\beta}^{(\nu)}(t-r,x-z)
	\sigma(u_\nu(r,z))
	W(\mathrm dr,\mathrm dz),
\end{equation}
where the Fourier transform of the Green kernel is
\begin{equation}
	\label{eq:kernel-fourier-introduction}
	\mathcal F
	G_{\alpha,\beta}^{(\nu)}(t,\cdot)(\xi)
	=
	E_\beta\left(-\nu t^\beta|\xi|^\alpha\right).
\end{equation}
Existence, uniqueness, moment bounds, and sample-path regularity for
equations of this type were investigated in
\cite{mijena2015space,foondun2017asymptotic}. In particular, the
Mittag--Leffler structure in
\eqref{eq:kernel-fourier-introduction}, rather than a semigroup
property, governs the temporal behavior of the solution.

The present paper studies the high-frequency temporal behavior of
\(u_\nu\). Fix \(0<T_1<T_2<\infty\) and \(x\in\mathbb R\), and consider
the equidistant observation times
\begin{equation}
	\label{eq:partition-introduction}
	t_i
	:=
	T_1+i\delta,
	\qquad
	\delta
	:=
	\frac{T_2-T_1}{n},
	\qquad
	i=0,1,\ldots,n.
\end{equation}
The restriction \(T_1>0\) keeps the observation interval away from
the singular initial boundary and is needed for the uniform temporal
increment estimates used below. For \(\gamma\geq2\), our principal
object is the normalized temporal power variation
\begin{equation}
	\label{eq:gamma-variation-introduction}
	n^{
		\frac{\gamma}{2}
		\left(1-\frac{\beta}{\alpha}\right)-1
	}
	\sum_{i=1}^n
	\left|
	u_\nu(t_i,x)-u_\nu(t_{i-1},x)
	\right|^\gamma.
\end{equation}
Besides determining its exact limit, we investigate its spatially
averaged counterpart and use the resulting asymptotic formulas to
construct parameter estimators.

Power variations have long provided an effective means of describing
the fine sample-path behavior of stochastic heat equations and of
constructing estimators from discrete observations. For the classical
stochastic heat equation, Posp\'i\v{s}il and
Tribe~\cite{pospisil2007parameter} obtained exact spatial quadratic
and exact temporal quartic variations and applied them to parameter
estimation, while Swanson~\cite{Swanson} studied several variation
functionals of the linear stochastic heat equation. More broadly,
Chong~\cite{chong2020high} established limit theorems for multipower
variations of parabolic SPDEs under high-frequency temporal
observations and used them for nonparametric volatility estimation.
Bibinger and Trabs~\cite{Bibinger-central,Bibinger-volatility}
obtained central limit theorems for power variations and developed
volatility estimation procedures based on discrete observations.
Parameter estimation from temporal and spatial increments was further
investigated in \cite{cialenco2020note,Hildebrandt2021}. For
discretely observed semilinear SPDEs, Cialenco et al. ~\cite{cialenco2024power} developed a
\(\Delta\)-power-variation approach based on higher-order finite
differences and applied it to drift and volatility estimation.
General surveys of statistical inference for SPDEs can be found in
\cite{cialenco2018,Lototsky2017}, and the monograph
\cite{tudor2022stochastic} provides a systematic account of variation
methods for stochastic heat equations.

Related variation problems have also been considered for equations
with fractional spatial operators or fractional Gaussian forcing.
Mahdi Khalil and Tudor~\cite{Mahdi-khalil-frac} constructed
power-variation estimators for the drift parameter of a stochastic
heat equation with a fractional Laplacian and additive Gaussian noise.
Torres, Tudor, and Viens~\cite{torres2014quadratic} studied quadratic
variations for a linear heat equation driven by fractional-colored
noise. In the multiplicative-noise setting, Gamain and
Tudor~\cite{gamain2022023exact} and Li et
al.~\cite{li2026exact} established spatial quadratic variation
results and corresponding drift estimators for nonlinear fractional
stochastic heat equations. More recently, Olivera and
Tudor~\cite{Oliver-frac} and Li et al.~\cite{Li2025} studied temporal
quadratic variations for nonlinear stochastic heat
equations involving a fractional Laplacian operator. Variation methods have
also been developed for stochastic heat equations with piecewise
constant coefficients: Li et
al.~\cite{li2024temporal,li2025spatial} established temporal quartic
and spatial quadratic variation results in the multiplicative-noise
setting, while Zili and Zougar~\cite{zili2019exact} considered the
additive-noise case. The fractional-Laplacian models cited above
retain a first-order time derivative. In contrast, equation
\eqref{eq:main-equation} contains both a Caputo time derivative and a
fractional spatial operator, so its temporal scaling depends jointly
on \(\alpha\) and \(\beta\).

The closest precursor to the present work is \cite{LI}, where temporal
variations were studied for the additive version of
\eqref{eq:main-equation}. More precisely, let \(u_0\) denote the mild
solution corresponding to \(\nu=1\) and \(\sigma\equiv1\), and define
the temporal increment by
\[
\Delta(t,\delta)
:=
u_0(t+\delta,x)-u_0(t,x).
\]
A fundamental estimate obtained in \cite{LI} is
\begin{equation}
	\label{eq:Delta2}
	\mathbb E
	\left|
	\Delta(t,\delta)
	\right|^2
	=
	C_{\alpha,\beta}
	\delta^{1-\beta/\alpha}
	+
	O(\delta^2),
	\qquad
	\delta\downarrow0,
\end{equation}
uniformly for \(t\) in compact intervals bounded away from zero.
Consequently,
\(\delta^{(1-\beta/\alpha)/2}\) is the natural
\(L^2(\Omega)\)-scale of a temporal increment.

The passage from additive to multiplicative noise introduces a
substantial difficulty. When \(\sigma\) is nonconstant, a temporal
increment of \(u_\nu\) is no longer Gaussian, and its random
coefficient depends on the same noise that generates the increment.
Thus, the Gaussian moment calculations available in the additive
case cannot be applied directly. To overcome this dependence, we
introduce a frozen-time approximation. The coefficient is frozen at
a time preceding the increment by a suitably chosen power of the mesh
size, while the part of the driving noise before that time is replaced
by an independent copy. The resulting modified increment has the same
distribution as an additive Gaussian increment and is independent of
the frozen coefficient. Estimates for the fractional Green kernel,
together with the temporal and spatial regularity of the solution,
then yield a quantitative \(L^2(\Omega)\)-bound for the approximation
error.

Our first main result shows that, for every real \(\gamma\geq2\),
\begin{align}
	&n^{
		\frac{\gamma}{2}
		\left(1-\frac{\beta}{\alpha}\right)-1
	}
	\sum_{i=1}^n
	\left|
	u_\nu(t_i,x)-u_\nu(t_{i-1},x)
	\right|^\gamma
	\notag\\
	&\quad\xrightarrow{L^1(\Omega)}
	\mu_\gamma
	C_{\alpha,\beta}^{\gamma/2}
	\nu^{-\gamma/(2\alpha)}
	(T_2-T_1)^{
		\frac{\gamma}{2}
		\left(1-\frac{\beta}{\alpha}\right)-1
	}
	\int_{T_1}^{T_2}
	\left|
	\sigma(u_\nu(r,x))
	\right|^\gamma
	\,\mathrm dr,
	\label{eq:main-result-introduction}
\end{align}
where
\begin{equation}
	\label{eq:mu-gamma-introduction}
	\mu_\gamma
	:=
	\mathbb E|Z|^\gamma
	=
	\frac{
		2^{\gamma/2}
		\Gamma\left(\frac{\gamma+1}{2}\right)
	}{
		\sqrt{\pi}
	},
	\qquad
	Z\sim N(0,1).
\end{equation}
Here \(C_{\alpha,\beta}\) is the leading constant in
\eqref{eq:Delta2}. We also obtain an explicit polynomial convergence
rate that is uniform with respect to the spatial location \(x\). This
uniformity allows us to average \eqref{eq:gamma-variation-introduction}
over an increasing number of spatial observation points. In
particular, if \(x_j=j/m\), then the corresponding spatially averaged
variation converges in \(L^1(\Omega)\) to the same deterministic
constant multiplied by
\[
\int_0^1
\int_{T_1}^{T_2}
\left|
\sigma(u_\nu(t,x))
\right|^\gamma
\,\mathrm dt\,\mathrm dx,
\]
as \(n,m\to\infty\), without requiring a prescribed relation between
the two sampling frequencies.

The normalized quadratic variation is recovered from
\eqref{eq:main-result-introduction} by taking \(\gamma=2\). On the
other hand, when
\(\gamma=p_{\alpha,\beta}=2\alpha/(\alpha-\beta)\), both the
normalizing power of \(n\) and the power of \(T_2-T_1\) disappear,
yielding an unnormalized critical power variation. Thus, the quadratic
and critical-order variations are two distinguished members of the
general family \eqref{eq:main-result-introduction}, rather than
separate phenomena.

Finally, we apply the variation limits to statistical inference.
Assuming that \(\alpha\), \(\beta\), and \(\sigma\) are known, we
construct consistent estimators of the drift, or diffusivity,
parameter \(\nu\), both from temporal observations at a fixed spatial
point and from observations collected over an increasing spatial
grid. We also show that the logarithmic growth of the unnormalized
temporal power variation consistently identifies the ratio
\(\beta/\alpha\). Consequently, \(\beta\) can be estimated when
\(\alpha\) is known, and conversely \(\alpha\) can be estimated when
\(\beta\) is known. If both fractional orders are unknown, temporal
high-frequency observations alone identify their ratio but do not
separate the two parameters.

The remainder of the paper is organized as follows. Section
\ref{se:2} introduces the Walsh--Dalang stochastic integral, explains
the mild formulation of \eqref{eq:main-equation}, and collects the
regularity and fractional-kernel estimates required later. In Section
\ref{se:3}, we establish the frozen-time approximation, prove the
normalized temporal power-variation theorem, and derive its spatially
averaged version. Section \ref{sec:applications} develops the
consistent estimators for \(\nu\) and for the ratio
\(\beta/\alpha\).

	\section{Preliminaries}
	\label{se:2}
	
	In this section, we collect several notions and technical facts concerning
	the space--time fractional stochastic heat equation that will be used
	throughout the paper.
	
	Throughout this section, \(\nu>0\) is fixed. To simplify notation, we
	suppress the dependence on \(\nu\) and write
	\[
	u(t,x):=u_\nu(t,x),
	\qquad
	G_t(x):=G_{\alpha,\beta}^{(\nu)}(t,x).
	\]
	
	\subsection{Space--time white noise and the Walsh--Dalang stochastic integral}
	
	Let
	\[
	W=\{W(t,A),\ t\geq0,\ A\in\mathcal B_b(\mathbb R)\}
	\]
	be a centered Gaussian random field on \([0,\infty)\times\mathbb R\),
	defined on a complete probability space
	\((\Omega,\mathcal F,\mathbb P)\), with covariance
	\[
	\mathbb E[W(t,A)W(s,B)]
	=
	(t\wedge s)\lambda(A\cap B),
	\qquad A,B\in\mathcal B_b(\mathbb R),
	\]
	where \(\lambda\) denotes the Lebesgue measure on \(\mathbb R\), and
	\(\mathcal B_b(\mathbb R)\) is the collection of bounded Borel subsets of
	\(\mathbb R\). For \(t\geq0\), we denote by \(\mathcal F_t\) the
	\(\sigma\)-field generated by
	\[
	\{W(s,A):0\leq s\leq t,\ A\in\mathcal B_b(\mathbb R)\},
	\]
	augmented in the usual way.
	
	Equivalently, the formal derivative of \(W\), denoted by \(\dot W\), is a
	space--time white noise with covariance
	\[
	\mathbb E[\dot W(t,x)\dot W(s,y)]
	=
	\delta_0(t-s)\delta_0(x-y).
	\]
	In the sense of Walsh \cite{Walsh1986} and Dalang \cite{dalang}, the
	stochastic integral
	\[
	\int_0^t\int_{\mathbb R} X(s,y)\,W(ds,dy)
	\]
	is well defined whenever the random field
	\[
	\{X(t,x):t\geq0,\ x\in\mathbb R\}
	\]
	is jointly measurable, adapted to the filtration
	\(\{\mathcal F_t\}_{t\geq0}\), and satisfies
	\[
	\mathbb E\int_0^t\int_{\mathbb R}
	|X(s,y)|^2\,dy\,ds<\infty.
	\]
	Moreover, the integral satisfies the It\^o isometry
	\[
	\mathbb E\left[
	\left(
	\int_0^t\int_{\mathbb R}
	X(s,y)\,W(ds,dy)
	\right)^2
	\right]
	=
	\mathbb E\int_0^t\int_{\mathbb R}
	|X(s,y)|^2\,dy\,ds .
	\]
	
	\subsection{Mild solution}
	
	For \(0<\beta<1\), the Caputo fractional derivative, introduced by
	Caputo in the study of linear dissipation models
	\cite{caputo1967linear}, is defined by
	\begin{equation}
		\label{eq:Caputo-derivative}
		\partial_t^\beta f(t)
		:=
		\frac{1}{\Gamma(1-\beta)}
		\int_0^t
		(t-r)^{-\beta}f'(r)\,\mathrm dr
		=
		I_t^{1-\beta}f'(t),
	\end{equation}
	whenever the right-hand side is well defined. Its Laplace transform
	satisfies
	\begin{equation}
		\label{eq:Laplace-Caputo}
		\mathcal L_t
		\left[
		\partial_t^\beta f
		\right](\lambda)
		=
		\lambda^\beta
		\mathcal L_t[f](\lambda)
		-
		\lambda^{\beta-1}f(0),
		\qquad \lambda>0.
	\end{equation}
	
	For \(\vartheta>0\), the Riemann--Liouville fractional integral is
	defined by
	\begin{equation}
		\label{eq:RL-integral}
		I_t^\vartheta f(t)
		:=
		\frac{1}{\Gamma(\vartheta)}
		\int_0^t
		(t-r)^{\vartheta-1}f(r)\,\mathrm dr.
	\end{equation}
	For \(0<\vartheta<1\), the corresponding Riemann--Liouville
	fractional derivative is
	\begin{equation}
		\label{eq:RL-derivative}
		D_t^\vartheta f(t)
		:=
		\frac{\mathrm d}{\mathrm dt}
		I_t^{1-\vartheta}f(t).
	\end{equation}
	Under the usual regularity and integrability assumptions,
	\begin{equation}
		\label{eq:RL-left-inverse}
		D_t^\vartheta I_t^\vartheta f(t)
		=
		f(t).
	\end{equation}
	
	We next explain the occurrence of \(I_t^{1-\beta}\) in
	\eqref{eq:main-equation}. Consider first the deterministic
	inhomogeneous problem
	\begin{equation}
		\label{eq:deterministic-inhomogeneous}
		\left\{
		\begin{aligned}
			\partial_t^\beta v(t,x)
			&=
			-\nu(-\Delta)^{\alpha/2}v(t,x)
			+
			f(t,x),\\
			v(0,x)&=v_0(x).
		\end{aligned}
		\right.
	\end{equation}
	The time-fractional Duhamel principle developed in
	\cite{umarov2006fractional,umarov2007generalization,
		umarov2012duhamel} gives
	\begin{align}
		v(t,x)
		={}&
		\int_{\mathbb R}
		G_t(x-y)v_0(y)\,\mathrm dy
		\notag\\
		&+
		\int_0^t
		\int_{\mathbb R}
		G_{t-r}(x-y)
		D_r^{1-\beta}f(r,y)
		\,\mathrm dy\,\mathrm dr.
		\label{eq:fractional-duhamel}
	\end{align}
	
	Suppose that the forcing term has the form
	\[
	f(t,x)
	=
	I_t^{1-\beta}g(t,x).
	\]
	By \eqref{eq:RL-left-inverse},
	\[
	D_t^{1-\beta}f(t,x)
	=
	D_t^{1-\beta}
	I_t^{1-\beta}g(t,x)
	=
	g(t,x).
	\]
	Consequently, \eqref{eq:fractional-duhamel} reduces to
	\begin{align}
		v(t,x)
		={}&
		\int_{\mathbb R}
		G_t(x-y)v_0(y)\,\mathrm dy
		\notag\\
		&+
		\int_0^t
		\int_{\mathbb R}
		G_{t-r}(x-y)g(r,y)
		\,\mathrm dy\,\mathrm dr.
		\label{eq:duhamel-cancelled}
	\end{align}
	
	The preceding calculation is first carried out for a sufficiently
	regular deterministic forcing term. Formally taking
	\[
	g(t,x)
	=
	\sigma(u(t,x))\dot W(t,x)
	\]
	explains why the stochastic forcing in \eqref{eq:main-equation} is
	written as
	\[
	I_t^{1-\beta}
	\bigl[
	\sigma(u(t,x))\dot W(t,x)
	\bigr].
	\]
	At the level of space--time white noise, the preceding calculation is
	understood as motivation rather than as a pathwise application of the
	fractional operators. Following
	\cite{mijena2015space,foondun2017asymptotic}, the corresponding
	stochastic convolution provides the rigorous definition of the mild
	solution:
	\begin{equation}
		\label{eq:mild}
		u(t,x)
		=
		\int_0^t
		\int_{\mathbb R}
		G_{t-r}(x-y)
		\sigma(u(r,y))
		W(\mathrm dr,\mathrm dy).
	\end{equation}
	Here \(G_t\) is the fundamental solution of
	\[
	\partial_t^\beta G_t(x)
	=
	-\nu(-\Delta)^{\alpha/2}G_t(x),
	\]
	and its spatial Fourier transform is
	\begin{equation}
		\label{eq:fourier-G-pre}
		\widehat G_t(\xi)
		=
		E_\beta(-\nu t^\beta|\xi|^\alpha),
		\qquad t>0,\quad \xi\in\mathbb R.
	\end{equation}
	
	By the \(L^2(\mathbb R)\)-estimate for
	\(G_t\) in \cite{mijena2015space},
	\begin{align*}
		\int_0^t
		\int_{\mathbb R}
		G_{t-r}^2(x-y)
		\,\mathrm dy\,\mathrm dr
		&=
		\int_0^t
		\|G_{t-r}\|_{L^2(\mathbb R)}^2
		\,\mathrm dr
		\\
		&=
		C_{\alpha,\beta,\nu}^{*}
		\int_0^t
		(t-r)^{-\beta/\alpha}
		\,\mathrm dr
		\\
		&=
		\frac{
			C_{\alpha,\beta,\nu}^{*}
		}{
			1-\beta/\alpha
		}
		t^{1-\beta/\alpha}
		<\infty,
	\end{align*}
	since \(0<\beta/\alpha<1\). Together with the global Lipschitz
	continuity of \(\sigma\), this estimate permits the standard Picard
	iteration construction of a unique random-field solution; see
	\cite{mijena2015space,foondun2017asymptotic}.
	
	Moreover, since the initial condition is identically zero, the moment
	estimate in \cite[Theorem~2]{mijena2015space} implies that, for every
	\(p\geq2\) and \(T>0\), there exists a constant
	\(C_{T,p}<\infty\) such that
	\begin{equation}
		\label{eq:uniform-moment-bound}
		\sup_{0\leq t\leq T}
		\sup_{x\in\mathbb R}
		\mathbb E|u(t,x)|^p
		\leq
		C_{T,p}.
	\end{equation}
	In particular, by the linear growth of \(\sigma\),
	\begin{align*}
		&\mathbb E
		\int_0^t
		\int_{\mathbb R}
		G_{t-r}^2(x-y)
		|\sigma(u(r,y))|^2
		\,\mathrm dy\,\mathrm dr
		\\
		&\quad\leq
		C_T
		\int_0^t
		\int_{\mathbb R}
		G_{t-r}^2(x-y)
		\,\mathrm dy\,\mathrm dr
		<\infty.
	\end{align*}
	Hence the stochastic convolution in \eqref{eq:mild} is
	well defined in the Walsh--Dalang sense.
	
	The solution also satisfies the following increment estimate. Set
	\begin{equation}
		\label{eq:kappa-star}
		\kappa_*
		:=
		\min\left\{
		2,\,
		2\alpha-1,\,
		\frac{\alpha-\beta}{\beta}
		\right\}.
	\end{equation}
	It follows from the argument used in the proof of the spatial increment
	estimate in \cite{mijena2015space} that, for every \(p\geq2\),
	\(T>0\), and \(0<q<\kappa_*\), there exists a constant
	\(C_{T,p,q}<\infty\) such that
	\begin{equation}
		\label{eq:increment}
		\mathbb E|u(t,x)-u(s,y)|^p
		\leq
		C_{T,p,q}
		\left(
		|t-s|^{\frac{p}{2}
			\left(1-\frac{\beta}{\alpha}\right)}
		+
		|x-y|^{\frac{pq}{2}}
		\right)
	\end{equation}
	for all \(s,t\in[0,T]\) and \(x,y\in\mathbb R\).
	
	Consequently, \(u\) admits a modification whose sample paths are
	locally H\"older continuous in time with every exponent
	\[
	\theta_t
	<
	\frac12\left(1-\frac{\beta}{\alpha}\right),
	\]
	and locally H\"older continuous in space with every exponent
	\[
	\theta_x
	<
	\frac{\kappa_*}{2}.
	\]
	
	In particular, the time regularity exponent is governed by the ratio
	\(\beta/\alpha\). This is consistent with the natural scaling relation
	\[
	x\sim t^{\beta/\alpha}
	\]
	for the space--time fractional heat kernel.

	\subsection{The space--time fractional Green kernel}
	\label{subsec:fractional-green-kernel}
	
	Let \(G_t(x)\) denote the fundamental solution of the deterministic
	space--time fractional heat equation
	\[
	\partial_t^\beta G_t(x)
	=
	-\nu(-\Delta)^{\alpha/2}G_t(x),
	\qquad t>0,\ x\in\mathbb R,
	\]
	where \(\nu>0\), \(\beta\in(0,1)\), and \(1<\alpha\leq2\). The Fourier
	transform of \(G_t\) with respect to the space variable is given by
	\begin{equation}
		\label{eq:fourier-trans-fractional}
		\mathcal F G_t(\cdot)(\xi)
		=
		\widehat G_t(\xi)
		=
		E_\beta(-\nu t^\beta|\xi|^\alpha),
		\qquad t>0,\ \xi\in\mathbb R,
	\end{equation}
	where
	\[
	E_\beta(z)
	=
	\sum_{k=0}^\infty
	\frac{z^k}{\Gamma(1+\beta k)}
	\]
	is the Mittag--Leffler function. We refer to
	\cite{haubold2011mittag,gorenflo2020mittag} for its basic properties
	and asymptotic estimates. Since our analysis relies heavily on
	Fourier techniques, we recall the Parseval--Plancherel identity: for
	\(f,g\in L^2(\mathbb R)\),
	\begin{equation}
		\label{eq:PP-iden-fractional}
		\langle f,g\rangle_{L^2(\mathbb R)}
		=
		\frac{1}{2\pi}
		\int_{\mathbb R}
		(\mathcal F f)(\xi)
		\overline{(\mathcal F g)(\xi)}
		\,d\xi .
	\end{equation}
	
	We now collect several properties of the kernel \(G_t\) which will be used throughout the sequel. For more details, we refer to
	\cite{mijena2015space,foondun2017asymptotic,LI}.
	
	\begin{itemize}
		\item[(i)]
		Let \(p_\alpha^{(\nu)}(s,x)\) be the transition density of the symmetric
		\(\alpha\)-stable process with Fourier transform
		\[
		\mathcal F p_\alpha^{(\nu)}(s,\cdot)(\xi)
		=
		e^{-\nu s|\xi|^\alpha}.
		\]
		Let \(E_t\) be the inverse of a \(\beta\)-stable subordinator, with density
		\(f_{E_t}(s)\). Then \(G_t\) admits the subordination representation
		\begin{equation}
			\label{eq:subordination-G}
			G_t(x)
			=
			\int_0^\infty
			p_\alpha^{(\nu)}(s,x)f_{E_t}(s)\,ds .
		\end{equation}
		Consequently, \(G_t(x)\geq0\), \(G_t\) is symmetric in the space variable,
		and
		\[
		\int_{\mathbb R}G_t(x)\,dx=1.
		\]
		
		\item[(ii)]
		The following two-sided estimate for the Mittag--Leffler function will
		be used repeatedly; see
		\cite{mijena2015space,gorenflo2020mittag}. For \(z>0\),
		\begin{equation}
			\label{eq:ML-estimate}
			\frac{1}{1+\Gamma(1-\beta)z}
			\leq
			E_\beta(-z)
			\leq
			\frac{1}{1+\Gamma(1+\beta)^{-1}z}.
		\end{equation}
		
		\item[(iii)]
		Unlike the classical heat kernel, \(G_t\) does not satisfy the semigroup
		property in general. Indeed,
		\[
		\widehat G_{t+s}(\xi)
		=
		E_\beta(-\nu(t+s)^\beta|\xi|^\alpha),
		\]
		whereas
		\[
		\widehat G_t(\xi)\widehat G_s(\xi)
		=
		E_\beta(-\nu t^\beta|\xi|^\alpha)
		E_\beta(-\nu s^\beta|\xi|^\alpha).
		\]
		These two quantities are not equal unless \(\beta=1\). Therefore, for
		\(\beta\in(0,1)\), one cannot use the classical semigroup identity
		\[
		G_{t+s}=G_t*G_s.
		\]
		
		\item[(iv)]
		The kernel \(G_t\) has the following scaling property. Let \(G_1^{(1)}\)
		denote the kernel corresponding to \(t=1\) and \(\nu=1\). Then
		\begin{equation}
			\label{eq:G-scaling-fractional}
			G_t^{(\nu)}(x)
			=
			(\nu t^\beta)^{-1/\alpha}
			G_1^{(1)}
			\left(
			(\nu t^\beta)^{-1/\alpha}x
			\right).
		\end{equation}
		Equivalently, its natural space-time scaling is
		\[
		x\sim t^{\beta/\alpha}.
		\]
		
		\item[(v)]
Since \(d=1\) and \(\alpha>1\), the following pointwise upper
estimate holds:
\begin{equation}
	\label{eq:G-pointwise-fractional}
	G_t(x)
	\leq
	C_{\alpha,\beta,\nu}
	\left(
	t^{-\beta/\alpha}
	\wedge
	\frac{t^\beta}{|x|^{1+\alpha}}
	\right),
	\qquad t>0,\quad x\in\mathbb R.
\end{equation}
In particular,
\[
G_t(x)
\leq
C_{\alpha,\beta,\nu}t^{-\beta/\alpha}.
\]
		
		\item[(vi)]
		For every \(t>0\),
		\begin{equation}
			\label{eq:L2-G-fractional}
			\int_{\mathbb R}G_t^2(x)\,dx
			=
			C_{\alpha,\beta,\nu}^{*}t^{-\beta/\alpha},
		\end{equation}
		where
		\begin{equation}
			\label{eq:C-star-fractional}
			C_{\alpha,\beta,\nu}^{*}
			=
			\frac{\nu^{-1/\alpha}}{\pi\alpha}
			\int_0^\infty
			z^{1/\alpha-1}
			\bigl(E_\beta(-z)\bigr)^2\,dz .
		\end{equation}
		
		\item[(vii)]
		For \(r>0\), regard \(G_r(\cdot)\) as an
		\(L^2(\mathbb R)\)-valued function. The derivative formula for the
		Mittag--Leffler function gives
		\[
		\frac{\partial}{\partial r}\widehat G_r(\xi)
		=
		-\nu|\xi|^\alpha r^{\beta-1}
		E_{\beta,\beta}(-\nu r^\beta|\xi|^\alpha);
		\]
		see, for example,
		\cite{haubold2011mittag,gorenflo2020mittag}.
		Moreover,
		\[
		|E_{\beta,\beta}(-z)|
		\leq
		\frac{C}{(1+z)^2},
		\qquad z\geq0.
		\]
		Consequently, by the Parseval--Plancherel identity,
		\begin{equation}
			\label{eq:time-derivative-G-L2}
			\left\|
			\frac{\partial}{\partial r}G_r(\cdot)
			\right\|_{L^2(\mathbb R)}^2
			\leq
			C_{\alpha,\beta,\nu}
			r^{-2-\beta/\alpha}.
		\end{equation}
		The details of this estimate were established in
		\cite[Lemma~1]{LI}.
	\end{itemize}
	The following temporal difference estimate for the fractional heat
	kernel was established in \cite{LI} and will be used repeatedly in the
	sequel.
	
	\begin{lemma}
		\label{lem:kernel-temporal-difference}
		There exists a constant \(C_{\alpha,\beta,\nu}<\infty\) such that, for
		every \(t>0\) and \(\delta>0\),
		\begin{equation}
			\label{eq:kernel-temporal-difference}
			\int_{\mathbb R}
			\left|
			G_{t+\delta}(y)
			-
			G_{t}(y)
			\right|^2
			\,\mathrm dy
			\leq
			C_{\alpha,\beta,\nu}\,
			\delta^2\,t^{-2-\beta/\alpha}.
		\end{equation}
	\end{lemma}
	\begin{proof}
		See \cite[Lemma~1]{LI}.
	\end{proof}
	Throughout this paper, \(C\) denotes a finite positive constant whose
	value may change from line to line. Dependence relevant to an argument
	is indicated by subscripts, while dependence on fixed model parameters
	is usually suppressed. Unless otherwise stated, these constants are
	independent of the spatial and temporal mesh sizes and of the grid
	indices.

\section{Temporal variations}\label{se:3}
Throughout this section, we set \(\nu=1\). We first establish a local
Gaussian approximation for the nonlinear temporal increments. We then
derive the normalized temporal \(\gamma\)-variation and its quadratic
and critical-order special cases. Finally, we investigate the
corresponding spatially averaged temporal variations.

\subsection{Frozen--time approximation}
\label{subsec:local-gaussian-approximation}

The mild solution of \eqref{eq:main-equation} with \(\nu=1\) is given
by
\begin{equation}
	\label{eq:mild-solution}
	u(t,x)
	=
	\int_0^t
	\int_{\mathbb R}
	G_{\alpha,\beta}(t-r,x-z)
	\sigma(u(r,z))
	W(\mathrm dr,\mathrm dz).
\end{equation}

Fix \(x\in\mathbb R\). For \(t>0\) and \(\delta>0\), define the
nonlinear temporal increment by
\begin{equation}
	\label{eq:nonlinear-temporal-increment}
	\bar{\Delta}(t,\delta)
	:=
	u(t+\delta,x)-u(t,x).
\end{equation}

Let \(\widetilde W\) be a space--time white noise independent of \(W\).
We define the modified temporal increment by
\begin{align}
	\widetilde{\Delta}(t,\delta)
	:={}&
	\int_0^{t(\delta)}
	\int_{\mathbb R}
	\left[
	G_{\alpha,\beta}(t+\delta-r,x-z)
	-
	G_{\alpha,\beta}(t-r,x-z)
	\right]
	\widetilde W(\mathrm dr,\mathrm dz)
	\notag\\
	&+
	\int_{t(\delta)}^{t+\delta}
	\int_{\mathbb R}
	\left[
	G_{\alpha,\beta}(t+\delta-r,x-z)
	-
	G_{\alpha,\beta}(t-r,x-z)
	\right]
	W(\mathrm dr,\mathrm dz),
	\label{eq:modified-temporal-increment}
\end{align}
where $t(\delta)
:=
t-\delta^{\kappa}$ provided $0<\delta^\kappa<1\wedge t$ and
\begin{equation}
	\kappa
	:=
	\frac{2\alpha}
	{2\alpha+\beta\kappa_*}.
\end{equation}

We use the convention
\[
G_{\alpha,\beta}(s,\cdot)=0,
\qquad s\leq0.
\]

Let \(u_0\) denote the mild solution of the corresponding additive
equation, and set
\[
\Delta(t,\delta)
:=
u_0(t+\delta,x)-u_0(t,x).
\]
Since the space--time white noises used on the two sides of
\(t(\delta)\) are independent, their concatenation is again a
space--time white noise. Therefore,
\[
\widetilde{\Delta}(t,\delta)
\overset{d}{=}
\Delta(t,\delta).
\]
Moreover, 
$\widetilde{\Delta}(t,\delta)$ 
 is independent of 
$\mathcal F_{t(\delta)}$.

At the classical endpoint \((\alpha,\beta)=(2,1)\), the above
formula gives \(\kappa_*=1\) and hence \(\kappa=4/5\), in agreement
with the exponent used for the classical stochastic heat equation
in~\cite{pospisil2007parameter}.

The following result provides a local Gaussian approximation of the
nonlinear temporal increment and will be used repeatedly in the sequel.
\begin{lemma}
	\label{lem:local-linearization}
	Set
	\begin{equation}
		\rho
		:=
		\frac{
			2[\alpha+\beta(\kappa_*-1)]
		}{
			2\alpha+\beta\kappa_*
		}.
	\end{equation}
	For any \(T>0\) and every
	\[
	0<\varepsilon<
	\min\left\{
	\rho,\,
	\frac{2\kappa\beta\kappa_*}{\alpha}
	\right\},
	\] there exists a constant
	\(C_{\varepsilon,T}<\infty\) such that
	\begin{equation}\label{eq:local-linearization}
		\mathbb E\left[
		\left|
		\bar{\Delta}(t,\delta)
		-
		\sigma\bigl(u(t(\delta),x)\bigr)
		\widetilde{\Delta}(t,\delta)
		\right|^2
		\right]
		\leq
		C_{\varepsilon,T}\,
		\delta^{\rho-\varepsilon}
	\end{equation}
	holds uniformly for all $x\in\mathbb R, 0< \delta^\kappa<t\wedge1$ and $t\leq T$.
\end{lemma}
\begin{proof}
	
	By the mild representation of \(u\), the definition of
	\(\bar\Delta(t,\delta)\), and
	\eqref{eq:modified-temporal-increment}, we have
	\begin{align*}
		&\bar{\Delta}(t,\delta)
		-
		\sigma\bigl(u(t(\delta),x)\bigr)
		\widetilde{\Delta}(t,\delta)
		\\
		&=
		\int_0^{t(\delta)}
		\int_{\mathbb R}
		\left[
		G_{\alpha,\beta}(t+\delta-r,x-z)
		-
		G_{\alpha,\beta}(t-r,x-z)
		\right]
		\sigma(u(r,z))
		W(\mathrm dr,\mathrm dz)
		\\
		&\quad+
		\int_{t(\delta)}^{t+\delta}
		\int_{\mathbb R}
		\left[
		G_{\alpha,\beta}(t+\delta-r,x-z)
		-
		G_{\alpha,\beta}(t-r,x-z)
		\right]
		\\
		&\hspace{38mm}\times
		\left[
		\sigma(u(r,z))
		-
		\sigma\bigl(u(t(\delta),x)\bigr)
		\right]
		W(\mathrm dr,\mathrm dz)
		\\
		&\quad-
		\sigma\bigl(u(t(\delta),x)\bigr)
		\int_0^{t(\delta)}
		\int_{\mathbb R}
		\left[
		G_{\alpha,\beta}(t+\delta-r,x-z)
		-
		G_{\alpha,\beta}(t-r,x-z)
		\right]
		\widetilde W(\mathrm dr,\mathrm dz)
		\\
		&=:I_1+I_2-I_3.
	\end{align*}
	Consequently,
	\begin{equation}\label{eq:three-error-terms}
		\mathbb E\left[
		\left|
		\bar{\Delta}(t,\delta)
		-
		\sigma\bigl(u(t(\delta),x)\bigr)
		\widetilde{\Delta}(t,\delta)
		\right|^2
		\right]\leq C\sum_{i=1}^{3}\mathbb{E}|I_i|^2.
	\end{equation}
	
	We first estimate \(I_1\). By the It\^o isometry, the Lipschitz continuity
	of \(\sigma\), and the uniform moment bound of the solution,
	\begin{align*}
		\mathbb E|I_1|^2
		&\leq
		C
		\int_0^{t(\delta)}
		\int_{\mathbb R}
		\left[
		G_{\alpha,\beta}(t+\delta-r,x-z)
		-
		G_{\alpha,\beta}(t-r,x-z)
		\right]^2
		\,\mathrm dz\,\mathrm dr.
	\end{align*}
	Making the change of variables \(s=t-r\) and using
	Lemma \ref{lem:kernel-temporal-difference}, we obtain
	\begin{align}
		\mathbb E|I_1|^2
		&\leq
		C\delta^2
		\int_{\delta^\kappa}^t
		s^{-2-\beta/\alpha}\,\mathrm ds
		\notag\\
		&\leq
		C
		\delta^{2-\kappa(1+\beta/\alpha)}
		=
		C\delta^\rho.\label{eq:I1-final-bound}
	\end{align}
	
	For \(I_3\), note that
	\(\sigma(u(t(\delta),x))\) is
	\(\mathcal F_{t(\delta)}\)-measurable, whereas
	\(\widetilde W\) is independent of
	\(\mathcal F_{t(\delta)}\). Hence, by the It\^o isometry and the
	uniform moment bound of the solution,
	\begin{align*}
		\mathbb E|I_3|^2
		&=
		\mathbb E\left|
		\sigma\bigl(u(t(\delta),x)\bigr)
		\right|^2
		\\
		&\quad\times
		\int_0^{t(\delta)}
		\int_{\mathbb R}
		\left[
		G_{\alpha,\beta}(t+\delta-r,x-z)
		-
		G_{\alpha,\beta}(t-r,x-z)
		\right]^2
		\,\mathrm dz\,\mathrm dr.
	\end{align*}
	The same calculation as for \(I_1\) gives
	\begin{equation}
		\label{eq:I3-final-bound}
		\mathbb E|I_3|^2
		\leq
		C\delta^\rho.
	\end{equation}
	
	It remains to estimate \(I_2\). For \(t(\delta)\leq r\leq t+\delta\),
	\[
	0\leq t(\delta)\leq r\leq t+\delta\leq T+1.
	\]
	Hence, in the following estimates we apply the moment and increment
	bounds on the interval \([0,T+1]\), while continuing to denote the
	resulting constants by \(C_{\varepsilon,T}\). Fix
	\[
	0<\varepsilon<
	\min\left\{
	\rho,\,
	\frac{2\kappa\beta\kappa_*}{\alpha}
	\right\}.
	\] Then
	\[
	0<
	\kappa_*-\frac{\alpha\varepsilon}{2\kappa\beta}
	<
	\kappa_*.
	\]
	Since the spatial increment estimate holds for every exponent strictly
	smaller than \(\kappa_*\), we may apply \eqref{eq:increment} with
	\(p=2\) and spatial exponent
	\[
	\kappa_*-\frac{\alpha\varepsilon}{2\kappa\beta}.
	\]
	By the Lipschitz continuity of \(\sigma\), we obtain
	\begin{align}
		&\mathbb E\left|
		\sigma(u(r,z))
		-
		\sigma\bigl(u(t(\delta),x)\bigr)
		\right|^2
		\notag\\
		&\leq
		C\mathbb E
		\left|u(r,z)-u(t(\delta),x)\right|^2
		\notag\\
		&\leq
		C_{\varepsilon,T}
		\left[
		|r-t(\delta)|^{1-\beta/\alpha}
		+
		|z-x|^{
			\kappa_*-\frac{\alpha\varepsilon}{2\kappa\beta}
		}
		\right].
		\label{eq:mixed-increment-bound}
	\end{align}
	For \(t(\delta)\leq r\leq t+\delta\), we have
	\[
	0\leq
	r-t(\delta)
	\leq
	\delta^\kappa+\delta.
	\]
	Since \(0<\delta\leq1\) and \(0<\kappa<1\), we have
	\(\delta\leq\delta^\kappa\), and hence
	\begin{equation}
		\label{eq:frozen-time-increment-bound}
		|r-t(\delta)|^{1-\beta/\alpha}
		\leq
		C\delta^{\kappa(1-\beta/\alpha)}.
	\end{equation}
	
	By the It\^o isometry, \eqref{eq:mixed-increment-bound}, and inequality
	\[
	(a-b)^2\leq2a^2+2b^2,
	\]
	we obtain
	\begin{align*}
		\mathbb E|I_2|^2
		&\leq
		C_{\varepsilon,T}
		\int_{t(\delta)}^{t+\delta}
		\int_{\mathbb R}
		G_{\alpha,\beta}^2(t+\delta-r,x-z)
		\\
		&\hspace{24mm}\times
		\left[
		\delta^{\kappa(1-\beta/\alpha)}
		+
		|z-x|^{
			\kappa_*-\frac{\alpha\varepsilon}{2\kappa\beta}
		}
		\right]
		\,\mathrm dz\,\mathrm dr
		\\
		&\quad+
		C_{\varepsilon,T}
		\int_{t(\delta)}^t
		\int_{\mathbb R}
		G_{\alpha,\beta}^2(t-r,x-z)
		\notag\\
		&\hspace{24mm}\times
		\left[
		\delta^{\kappa(1-\beta/\alpha)}
		+
		|z-x|^{
			\kappa_*-\frac{\alpha\varepsilon}{2\kappa\beta}
		}
		\right]
		\,\mathrm dz\,\mathrm dr.
	\end{align*}
	
	By \eqref{eq:L2-G-fractional},
	we have
	\begin{align*}
		&\delta^{\kappa(1-\beta/\alpha)}
		\int_{t(\delta)}^{t+\delta}
		\int_{\mathbb R}
		G_{\alpha,\beta}^2(t+\delta-r,x-z)
		\,\mathrm dz\,\mathrm dr
		\\
		&\quad+
		\delta^{\kappa(1-\beta/\alpha)}
		\int_{t(\delta)}^t
		\int_{\mathbb R}
		G_{\alpha,\beta}^2(t-r,x-z)
		\,\mathrm dz\,\mathrm dr
		\\
		&\leq
		C\delta^{\kappa(1-\beta/\alpha)}
		\left[
		\int_0^{\delta^\kappa+\delta}
		s^{-\beta/\alpha}\,\mathrm ds
		+
		\int_0^{\delta^\kappa}
		s^{-\beta/\alpha}\,\mathrm ds
		\right]
		\\
		&\leq
		C\delta^{\kappa(1-\beta/\alpha)}
		\left(\delta^\kappa\right)^{1-\beta/\alpha}
		\\
		&=
		C\delta^{2\kappa(1-\beta/\alpha)}.
	\end{align*}
	
	We next estimate the contribution arising from the spatial increment.
	By the scaling property \eqref{eq:G-scaling-fractional}, for every
	\(s>0\),
	\begin{align*}
		&\int_{\mathbb R}
		G_{\alpha,\beta}^2(s,y)
		|y|^{
			\kappa_*-\frac{\alpha\varepsilon}{2\kappa\beta}
		}
		\,\mathrm dy
		\\
		&=
		s^{-2\beta/\alpha}
		\int_{\mathbb R}
		G_{\alpha,\beta}^2
		\left(
		1,s^{-\beta/\alpha}y
		\right)
		|y|^{
			\kappa_*-\frac{\alpha\varepsilon}{2\kappa\beta}
		}
		\,\mathrm dy
		\\
		&=
		s^{
			\frac{\beta}{\alpha}
			\left(
			\kappa_*-\frac{\alpha\varepsilon}{2\kappa\beta}-1
			\right)
		}
		\int_{\mathbb R}
		G_{\alpha,\beta}^2(1,w)
		|w|^{
			\kappa_*-\frac{\alpha\varepsilon}{2\kappa\beta}
		}
		\,\mathrm dw,
	\end{align*}
	where in the last equality we used the change of variables
	\(
	w=s^{-\beta/\alpha}y.
	\)
	
	Since \(\varepsilon>0\) is sufficiently small,
	\[
	0<
	\kappa_*-\frac{\alpha\varepsilon}{2\kappa\beta}
	<
	\kappa_*
	\leq2.
	\]
	Moreover, by the pointwise kernel estimate
	\eqref{eq:G-pointwise-fractional},
	\[
	G_{\alpha,\beta}(1,w)
	\leq
	C_{\alpha,\beta}
	\left(
	1\wedge\frac{1}{|w|^{1+\alpha}}
	\right).
	\]
	It follows that
	\begin{align*}
		&\int_{\mathbb R}
		G_{\alpha,\beta}^2(1,w)
		|w|^{
			\kappa_*-\frac{\alpha\varepsilon}{2\kappa\beta}
		}
		\,\mathrm dw
		\\
		&\leq
		C_{\alpha,\beta}
		\int_{|w|\leq1}
		|w|^{
			\kappa_*-\frac{\alpha\varepsilon}{2\kappa\beta}
		}
		\,\mathrm dw
		\\
		&\qquad+
		C_{\alpha,\beta}
		\int_{|w|>1}
		|w|^{
			\kappa_*-\frac{\alpha\varepsilon}{2\kappa\beta}
			-2-2\alpha
		}
		\,\mathrm dw.
	\end{align*}
	Indeed, the first integral is finite because
	\[
	\kappa_*-\frac{\alpha\varepsilon}{2\kappa\beta}>0,
	\]
	whereas the second one is finite because
	\[
	\kappa_*-\frac{\alpha\varepsilon}{2\kappa\beta}
	<\kappa_*\leq2<2\alpha+1.
	\]
	Consequently,
	\begin{equation}
		\label{eq:weighted-kernel-in-proof}
		\int_{\mathbb R}
		G_{\alpha,\beta}^2(s,y)
		|y|^{
			\kappa_*-\frac{\alpha\varepsilon}{2\kappa\beta}
		}
		\,\mathrm dy
		\leq
		C_{\varepsilon}
		s^{
			\frac{\beta}{\alpha}
			\left(
			\kappa_*-\frac{\alpha\varepsilon}{2\kappa\beta}-1
			\right)
		}.
	\end{equation}
	
	Applying \eqref{eq:weighted-kernel-in-proof}, together with the changes
	of variables
	\[
	s=t+\delta-r
	\qquad\text{and}\qquad
	s=t-r,
	\]
	respectively, we obtain
	\begin{align*}
		&\int_{t(\delta)}^{t+\delta}
		\int_{\mathbb R}
		G_{\alpha,\beta}^2(t+\delta-r,x-z)
		|z-x|^{
			\kappa_*-\frac{\alpha\varepsilon}{2\kappa\beta}
		}
		\,\mathrm dz\,\mathrm dr
		\\
		&\quad+
		\int_{t(\delta)}^t
		\int_{\mathbb R}
		G_{\alpha,\beta}^2(t-r,x-z)
		|z-x|^{
			\kappa_*-\frac{\alpha\varepsilon}{2\kappa\beta}
		}
		\,\mathrm dz\,\mathrm dr
		\\
		&\leq
		C_{\varepsilon}
		\left[
		\int_0^{\delta^\kappa+\delta}
		s^{
			\frac{\beta}{\alpha}
			\left(
			\kappa_*-\frac{\alpha\varepsilon}{2\kappa\beta}-1
			\right)
		}
		\,\mathrm ds
		+
		\int_0^{\delta^\kappa}
		s^{
			\frac{\beta}{\alpha}
			\left(
			\kappa_*-\frac{\alpha\varepsilon}{2\kappa\beta}-1
			\right)
		}
		\,\mathrm ds
		\right].
	\end{align*}
	Notice that
	\[
	\frac{\beta}{\alpha}
	\left(
	\kappa_*-\frac{\alpha\varepsilon}{2\kappa\beta}-1
	\right)
	>
	-\frac{\beta}{\alpha}
	>-1.
	\]
	Therefore, both integrals are finite. Since \(0<\delta\leq1\) and
	\(0<\kappa<1\), we also have
	\[
	\delta\leq\delta^\kappa
	\qquad\text{and}\qquad
	\delta^\kappa+\delta\leq2\delta^\kappa.
	\]
	It follows that
	\begin{align*}
		&\int_{t(\delta)}^{t+\delta}
		\int_{\mathbb R}
		G_{\alpha,\beta}^2(t+\delta-r,x-z)
		|z-x|^{
			\kappa_*-\frac{\alpha\varepsilon}{2\kappa\beta}
		}
		\,\mathrm dz\,\mathrm dr
		\\
		&\quad+
		\int_{t(\delta)}^t
		\int_{\mathbb R}
		G_{\alpha,\beta}^2(t-r,x-z)
		|z-x|^{
			\kappa_*-\frac{\alpha\varepsilon}{2\kappa\beta}
		}
		\,\mathrm dz\,\mathrm dr
		\\
		&\leq
		C_{\varepsilon}
		\delta^{
			\kappa\left[
			1+
			\frac{\beta}{\alpha}
			\left(
			\kappa_*-\frac{\alpha\varepsilon}{2\kappa\beta}-1
			\right)
			\right]
		}.
	\end{align*}
	
	Hence, we obtain
	\begin{equation}
		\mathbb E|I_2|^2
		\leq
		C_{\varepsilon,T}
		\bigg[
		\delta^{2\kappa(1-\beta/\alpha)}
		+
		\delta^{
			\kappa\left[
			1+
			\frac{\beta}{\alpha}
			\left(
			\kappa_*-\frac{\alpha\varepsilon}{2\kappa\beta}-1
			\right)
			\right]
		}
		\bigg].
	\end{equation}
	
	Since
	\[
	\kappa_*-\frac{\alpha\varepsilon}{2\kappa\beta}
	<
	\kappa_*
	\leq
	\frac{\alpha-\beta}{\beta},
	\]
	we have
	\begin{align*}
		&1+
		\frac{\beta}{\alpha}
		\left(
		\kappa_*-\frac{\alpha\varepsilon}{2\kappa\beta}-1
		\right)
		\\
		&<
		1+
		\frac{\beta}{\alpha}
		\left(
		\frac{\alpha-\beta}{\beta}-1
		\right)
		\\
		&=
		2\left(1-\frac{\beta}{\alpha}\right).
	\end{align*}
	It follows that
	\begin{align}
		\mathbb E|I_2|^2
		&\leq
		C_{\varepsilon,T}
		\delta^{
			\kappa\left[
			1+
			\frac{\beta}{\alpha}
			\left(
			\kappa_*-\frac{\alpha\varepsilon}{2\kappa\beta}-1
			\right)
			\right]
		}.
		\label{eq:I2-dominant-order}
	\end{align}
	By the definitions of \(\rho\) and \(\kappa\),
	\begin{align*}
		&\kappa\left[
		1+
		\frac{\beta}{\alpha}
		\left(
		\kappa_*-\frac{\alpha\varepsilon}{2\kappa\beta}-1
		\right)
		\right]
		\\
		&\quad=
		\kappa\left(
		1+\frac{\beta(\kappa_*-1)}{\alpha}
		\right)
		-
		\frac{\varepsilon}{2}
		\\
		&\quad=
		\rho-\frac{\varepsilon}{2}.
	\end{align*}
	Consequently,
	\begin{equation}
		\label{eq:I2-general-bound}
		\mathbb E|I_2|^2
		\leq
		C_{\varepsilon,T}\,
		\delta^{\rho-\varepsilon/2}
		\leq
		C_{\varepsilon,T}\,
		\delta^{\rho-\varepsilon},
	\end{equation}
	where the last inequality follows from \(0<\delta\leq1\). This proves
	the required estimate for \(I_2\).

	Combining \eqref{eq:I1-final-bound}, \eqref{eq:I3-final-bound}, and \eqref{eq:I2-general-bound} yields the desired estimate and completes the proof.
\end{proof}

\subsection{Normalized temporal power variations}
\label{subsec:normalized-temporal-power-variations}

We now establish the normalized temporal \(\gamma\)-variation of the
solution.
\begin{theorem}
	\label{thm:normalized-temporal-gamma-variation}
	Let \(0<T_1<T_2<\infty\) and \(x\in\mathbb R\) be fixed. For each
	\(n\geq1\), set
	\[
	\delta:=\frac{T_2-T_1}{n},
	\qquad
	t_i:=T_1+i\delta,
	\qquad i=0,1,\ldots,n.
	\]
	Then, for every fixed \(\gamma\geq2\), we have
	\begin{align}
		&n^{
			\frac{\gamma}{2}
			\left(1-\frac{\beta}{\alpha}\right)-1
		}
		\sum_{i=1}^n
		\left|
		u(t_i,x)-u(t_{i-1},x)
		\right|^\gamma
		\notag\\
		&\quad\overset{L^1(\Omega)}{\longrightarrow}
		\mu_\gamma C_{\alpha,\beta}^{\gamma/2}
		(T_2-T_1)^{
			\frac{\gamma}{2}
			\left(1-\frac{\beta}{\alpha}\right)-1
		}
		\int_{T_1}^{T_2}
		|\sigma(u(r,x))|^\gamma
		\,\mathrm dr,
		\label{eq:normalized-temporal-gamma-variation}
	\end{align}
	where
	\[
	\mu_\gamma
	:=
	\mathbb E|Z|^\gamma
	=
	\frac{
		2^{\gamma/2}
		\Gamma\left(\frac{\gamma+1}{2}\right)
	}{
		\sqrt{\pi}
	},
	\qquad Z\sim N(0,1).
	\]
	
	Moreover, for every
	\begin{equation}
		\label{eq:choice-temporal-rate}
		0<\eta<
		\frac{
			\beta\kappa_*(\alpha+\beta)
		}{
			2\alpha(2\alpha+\beta\kappa_*)
		},
	\end{equation}
	there exists a constant
	\(C_{\alpha,\beta,\gamma,T_1,T_2,\eta}<\infty\) such that, for all
	sufficiently large \(n\),
	\begin{align}
		&\sup_{x\in\mathbb R}
		\mathbb E\bigg|
		n^{
			\frac{\gamma}{2}
			\left(1-\frac{\beta}{\alpha}\right)-1
		}
		\sum_{i=1}^n
		\left|
		u(t_i,x)-u(t_{i-1},x)
		\right|^\gamma
		\notag\\
		&\qquad-
		\mu_\gamma C_{\alpha,\beta}^{\gamma/2}
		(T_2-T_1)^{
			\frac{\gamma}{2}
			\left(1-\frac{\beta}{\alpha}\right)-1
		}
		\int_{T_1}^{T_2}
		|\sigma(u(r,x))|^\gamma
		\,\mathrm dr
		\bigg|
		\notag\\
		&\quad\leq
		C_{\alpha,\beta,\gamma,T_1,T_2,\eta}
		n^{-\eta}.
		\label{eq:normalized-temporal-gamma-rate}
	\end{align}
\end{theorem}

\begin{proof}
	Since \(T_1>0\) and \(\delta=(T_2-T_1)/n\), for all sufficiently
	large \(n\),
	\[
	\delta^\kappa<T_1\wedge1.
	\]
	For such \(n\), define
	\[
	t_{i-1}(\delta)
	:=
	t_{i-1}-\delta^\kappa,
	\qquad i=1,\ldots,n.
	\]
	Then
	\[
	0<t_{i-1}(\delta)<t_{i-1},
	\qquad i=1,\ldots,n.
	\]
	
	Enlarging the underlying probability space if necessary, let
\(\widetilde W_1,\ldots,\widetilde W_n\) be mutually independent
space--time white noises that are jointly independent of \(W\). For each \(i=1,\ldots,n\), let
	\(\widetilde\Delta(t_{i-1},\delta)\) be the modified temporal
	increment defined by \eqref{eq:modified-temporal-increment}, with
	\(\widetilde W\) replaced by \(\widetilde W_i\). Then
	\[
	\widetilde\Delta(t_{i-1},\delta)
	\overset{d}{=}
	\Delta(t_{i-1},\delta),
	\]
	and
	\[
	\widetilde\Delta(t_{i-1},\delta)
	\quad\text{is independent of}\quad
	\mathcal F_{t_{i-1}(\delta)}.
	\]
	Moreover, Lemma \ref{lem:local-linearization} can be applied
	uniformly with \(t=t_{i-1}\) and \(T=T_2\).
	
	We decompose the difference between the two sides of
	\eqref{eq:normalized-temporal-gamma-variation} as
	\[
	J_1+J_2+J_3,
	\]
	where
	\begin{equation*}
		J_1
		:=
		n^{
			\frac{\gamma}{2}
			\left(1-\frac{\beta}{\alpha}\right)-1
		}
		\sum_{i=1}^n
		\left[
		\left|
		u(t_i,x)-u(t_{i-1},x)
		\right|^\gamma
		-
		\left|
		\sigma\bigl(u(t_{i-1}(\delta),x)\bigr)
		\widetilde\Delta(t_{i-1},\delta)
		\right|^\gamma
		\right],
	\end{equation*}
	\begin{equation*}
		J_2
		:=
		n^{
			\frac{\gamma}{2}
			\left(1-\frac{\beta}{\alpha}\right)-1
		}
		\sum_{i=1}^n
		\left|
		\sigma\bigl(u(t_{i-1}(\delta),x)\bigr)
		\right|^\gamma
		\bigg[
		\left|
		\widetilde\Delta(t_{i-1},\delta)
		\right|^\gamma
		-
		\mu_\gamma C_{\alpha,\beta}^{\gamma/2}
		\delta^{
			\frac{\gamma}{2}
			\left(1-\frac{\beta}{\alpha}\right)
		}
		\bigg],
	\end{equation*}
	and
	\begin{align*}
		J_3
		:={}&
		\mu_\gamma C_{\alpha,\beta}^{\gamma/2}
		n^{
			\frac{\gamma}{2}
			\left(1-\frac{\beta}{\alpha}\right)-1
		}
		\delta^{
			\frac{\gamma}{2}
			\left(1-\frac{\beta}{\alpha}\right)
		}
		\sum_{i=1}^n
		\left|
		\sigma\bigl(u(t_{i-1}(\delta),x)\bigr)
		\right|^\gamma
		\\
		&-
		\mu_\gamma C_{\alpha,\beta}^{\gamma/2}
		(T_2-T_1)^{
			\frac{\gamma}{2}
			\left(1-\frac{\beta}{\alpha}\right)-1
		}
		\int_{T_1}^{T_2}
		|\sigma(u(r,x))|^\gamma
		\,\mathrm dr.
	\end{align*}
 

	We estimate the three terms separately.
	
	\medskip
	\noindent
	\textit{Estimate of \(J_1\).}
	
	Using the fact
	\[
	\left|
	|a|^\gamma-|b|^\gamma
	\right|
	\leq
	C_\gamma|a-b|
	\left(
	|a|^{\gamma-1}+|b|^{\gamma-1}
	\right),\quad a,b\in\mathbb{R},
	\]
	and the Cauchy--Schwarz inequality, we obtain
	\begin{align*}
		\mathbb E|J_1|
		\leq{}&
		C_\gamma
		n^{
			\frac{\gamma}{2}
			\left(1-\frac{\beta}{\alpha}\right)-1
		}
		\sum_{i=1}^n
		\bigg(
		\mathbb E
		\Big|
		\bar\Delta(t_{i-1},\delta)
		-
		\sigma\bigl(u(t_{i-1}(\delta),x)\bigr)
		\widetilde\Delta(t_{i-1},\delta)
		\Big|^2
		\bigg)^{1/2}
		\\
		&\quad\times
		\bigg[
		\left\|
		\bar\Delta(t_{i-1},\delta)
		\right\|_{2(\gamma-1)}^{\gamma-1}
		+
		\left\|
		\sigma\bigl(u(t_{i-1}(\delta),x)\bigr)
		\widetilde\Delta(t_{i-1},\delta)
		\right\|_{2(\gamma-1)}^{\gamma-1}
		\bigg].
	\end{align*}
	
	The temporal increment estimate gives
	\[
	\left\|
	\bar\Delta(t_{i-1},\delta)
	\right\|_{2(\gamma-1)}
	\leq
	C\delta^{\frac12(1-\frac\beta\alpha)}.
	\]

	By \eqref{eq:Delta2},
	\begin{equation}
		\label{eq:gamma-additive-second-moment}
		\mathbb E
		\left[
		\widetilde\Delta^2(t_{i-1},\delta)
		\right]
		=
		C_{\alpha,\beta}
		\delta^{1-\beta/\alpha}
		+
		O(\delta^2),
	\end{equation}
	uniformly in \(i=1,\ldots,n\).
	Then by the independence of
	\(\widetilde\Delta(t_{i-1},\delta)\) and
	\(\mathcal F_{t_{i-1}(\delta)}\), the Gaussian moment formula,
	and the uniform moment bounds of \(\sigma(u)\),
	\[
	\left\|
	\sigma\bigl(u(t_{i-1}(\delta),x)\bigr)
	\widetilde\Delta(t_{i-1},\delta)
	\right\|_{2(\gamma-1)}
	\leq
	C\delta^{\frac12(1-\frac\beta\alpha)}.
	\]
	
	Choose \(\varepsilon>0\) sufficiently small so that
	\[
	0<\varepsilon<
	\frac{
		\beta\kappa_*(\alpha+\beta)
	}{
		\alpha(2\alpha+\beta\kappa_*)
	}.
	\]
	By Lemma \ref{lem:local-linearization},
	\begin{equation*}
		\mathbb E
		\Big|
		\bar\Delta(t_{i-1},\delta)-
		\sigma\bigl(u(t_{i-1}(\delta),x)\bigr)
		\widetilde\Delta(t_{i-1},\delta)
		\Big|^2\leq
		C_{\varepsilon,T_2}\,
		\delta^{
			\frac{
				2[\alpha+\beta(\kappa_*-1)]
			}{
				2\alpha+\beta\kappa_*
			}
			-\varepsilon
		}.
	\end{equation*}
	It follows that
	\begin{equation*}
		\mathbb E|J_1|
		\leq
		C_{\varepsilon,T_2}\,
		n^{
			\frac{\gamma}{2}
			\left(1-\frac{\beta}{\alpha}\right)
		}\,
		\delta^{
			\frac12
			\left[
			\frac{
				2[\alpha+\beta(\kappa_*-1)]
			}{
				2\alpha+\beta\kappa_*
			}
			-\varepsilon
			\right]
			+
			\frac{\gamma-1}{2}
			\left(1-\frac{\beta}{\alpha}\right)
		}.
	\end{equation*}
	Since \(\delta=(T_2-T_1)/n\), we obtain
	\begin{equation}\label{eq:gamma-J1-final}
		\mathbb E|J_1|
		\leq
		C_{\varepsilon,T_1,T_2}\,
		n^{
			-\frac12
			\left[
			\frac{
				\beta\kappa_*(\alpha+\beta)
			}{
				\alpha(2\alpha+\beta\kappa_*)
			}
			-\varepsilon
			\right]
		}
		\longrightarrow0,\qquad n\to \infty,
	\end{equation}
	where the convergence follows from the choice of $\varepsilon$.
	
	\medskip
	\noindent
	\textit{Estimate of \(J_2\).}
	
	Expanding its second moment, we have
	\begin{align*}
		\mathbb E|J_2|^2
		={}&
		n^{
			\gamma\left(1-\frac{\beta}{\alpha}\right)-2
		}
		\sum_{i,j=1}^n
		\mathbb E\bigg[
		\left|
		\sigma\bigl(u(t_{i-1}(\delta),x)\bigr)
		\right|^\gamma
		\left|
		\sigma\bigl(u(t_{j-1}(\delta),x)\bigr)
		\right|^\gamma
		\notag\\
		&\quad\times
		\bigg(
		\left|
		\widetilde\Delta(t_{i-1},\delta)
		\right|^\gamma
		-
		\mu_\gamma C_{\alpha,\beta}^{\gamma/2}
		\delta^{
			\frac{\gamma}{2}
			\left(1-\frac{\beta}{\alpha}\right)
		}
		\bigg)
		\notag\\
		&\quad\times
		\bigg(
		\left|
		\widetilde\Delta(t_{j-1},\delta)
		\right|^\gamma
		-
		\mu_\gamma C_{\alpha,\beta}^{\gamma/2}
		\delta^{
			\frac{\gamma}{2}
			\left(1-\frac{\beta}{\alpha}\right)
		}
		\bigg)
		\bigg]\\
		=:&J_{2,1}+J_{2,2},
	\end{align*}
	where we split the preceding double sum into the near-diagonal and
far-off-diagonal regions and denote the corresponding contributions
by \(J_{2,1}\) and \(J_{2,2}\), respectively, where
\[
|i-j|
\leq
\left\lfloor\delta^{\kappa-1}\right\rfloor+1
\]
for the near-diagonal region and
\[
|i-j|
>
\left\lfloor\delta^{\kappa-1}\right\rfloor+1
\]
for the far-off-diagonal region.
	
	When $|i-j|
	\leq
	\left\lfloor\delta^{\kappa-1}\right\rfloor+1 $, H\"older's inequality yields
	\begin{align*}
		|J_{2,1}|
		\leq{}&n^{
			\gamma\left(1-\frac{\beta}{\alpha}\right)-2
		}
	\sum_{\substack{1 \le i,j \le n \\ |i-j| \le \lfloor \delta^{\kappa-1} \rfloor + 1}}\left(
		\mathbb E
		\left|
		\sigma\bigl(u(t_{i-1}(\delta),x)\bigr)
		\right|^{4\gamma}
		\right)^{1/4}
		\left(
		\mathbb E
		\left|
		\sigma\bigl(u(t_{j-1}(\delta),x)\bigr)
		\right|^{4\gamma}
		\right)^{1/4}
		\\
		&\quad\times
		\left\|
		\left|
		\widetilde\Delta(t_{i-1},\delta)
		\right|^\gamma
		-
		\mu_\gamma C_{\alpha,\beta}^{\gamma/2}
		\delta^{
			\frac{\gamma}{2}
			\left(1-\frac{\beta}{\alpha}\right)
		}
		\right\|_4
		\\
		&\quad\times
		\left\|
		\left|
		\widetilde\Delta(t_{j-1},\delta)
		\right|^\gamma
		-
		\mu_\gamma C_{\alpha,\beta}^{\gamma/2}
		\delta^{
			\frac{\gamma}{2}
			\left(1-\frac{\beta}{\alpha}\right)
		}
		\right\|_4.
	\end{align*}
	
	Using the Gaussian moment formula and
	\eqref{eq:gamma-additive-second-moment}, we have
	\begin{align}
		&\left\|
		\left|
		\widetilde\Delta(t_{i-1},\delta)
		\right|^\gamma
		-
		\mu_\gamma C_{\alpha,\beta}^{\gamma/2}
		\delta^{
			\frac{\gamma}{2}
			\left(1-\frac{\beta}{\alpha}\right)
		}
		\right\|_4
		\notag\\
		&\leq
		\left(
		\mathbb E
		\left|
		\widetilde\Delta(t_{i-1},\delta)
		\right|^{4\gamma}
		\right)^{1/4}
		+
		\mu_\gamma C_{\alpha,\beta}^{\gamma/2}
		\delta^{
			\frac{\gamma}{2}
			\left(1-\frac{\beta}{\alpha}\right)
		}
		\notag\\
		&\leq
		C
		\left(
		\mathbb E
		\widetilde\Delta^2(t_{i-1},\delta)
		\right)^{\gamma/2}
		+
		C
		\delta^{
			\frac{\gamma}{2}
			\left(1-\frac{\beta}{\alpha}\right)
		}
		\notag\\
		&\leq
		C
		\delta^{
			\frac{\gamma}{2}
			\left(1-\frac{\beta}{\alpha}\right)
		}.
		\label{eq:gamma-centered-fourth-moment}
	\end{align}
	Therefore, by Lipschitz continuity of $\sigma$, the uniform moment bounds for \(u\), and \eqref{eq:gamma-centered-fourth-moment}, we have 
	\begin{equation}\label{eq:gamma-J2-near-final}
		|J_{2,1}|\leq
		C
		n^{
			\gamma\left(1-\frac{\beta}{\alpha}\right)-2
		}\,
		n\,\delta^{\kappa-1}
		\delta^{
			\gamma\left(1-\frac{\beta}{\alpha}\right)
		}
		\leq
		C\delta^\kappa.
	\end{equation}
	
	For $J_{2,2}$, suppose without loss of generality that
	\(i>j\) and
	\[
	i-j>
	\left\lfloor\delta^{\kappa-1}\right\rfloor+1.
	\]
	Then
	\[
	(i-j-1)\delta>\delta^\kappa,
	\]
	and hence
	\[
	t_j
	=
	t_{i-1}-(i-j-1)\delta
	<
	t_{i-1}(\delta).
	\]
	It follows that the random variables
	\[
	\sigma\bigl(u(t_{i-1}(\delta),x)\bigr),
	\qquad
	\sigma\bigl(u(t_{j-1}(\delta),x)\bigr),
	\qquad
	\widetilde\Delta(t_{j-1},\delta)
	\]
	are measurable with respect to
	\[
	\mathcal F_{t_{i-1}(\delta)}
	\vee
	\sigma(\widetilde W_j),
	\]
	where \(\sigma(\widetilde W_j)\) denotes the \(\sigma\)-algebra
	generated by the auxiliary white noise \(\widetilde W_j\).
	Moreover, \(\widetilde\Delta(t_{i-1},\delta)\) is independent of
	\(
	\mathcal F_{t_{i-1}(\delta)}
	\vee
	\sigma(\widetilde W_j).
	\)
	
	Consequently, we obtain
	\begin{align*}
		&\mathbb E\bigg[
		\left|
		\sigma\bigl(u(t_{i-1}(\delta),x)\bigr)
		\right|^\gamma
		\left|
		\sigma\bigl(u(t_{j-1}(\delta),x)\bigr)
		\right|^\gamma
		\\
		&\quad\times
		\bigg(
		\left|
		\widetilde\Delta(t_{i-1},\delta)
		\right|^\gamma
		-
		\mu_\gamma C_{\alpha,\beta}^{\gamma/2}
		\delta^{
			\frac{\gamma}{2}
			\left(1-\frac{\beta}{\alpha}\right)
		}
		\bigg)
		\\
		&\quad\times
		\bigg(
		\left|
		\widetilde\Delta(t_{j-1},\delta)
		\right|^\gamma
		-
		\mu_\gamma C_{\alpha,\beta}^{\gamma/2}
		\delta^{
			\frac{\gamma}{2}
			\left(1-\frac{\beta}{\alpha}\right)
		}
		\bigg)
		\bigg]
		\\
		&=
		\mathbb E\bigg[
		\left|
		\widetilde\Delta(t_{i-1},\delta)
		\right|^\gamma
		-
		\mu_\gamma C_{\alpha,\beta}^{\gamma/2}
		\delta^{
			\frac{\gamma}{2}
			\left(1-\frac{\beta}{\alpha}\right)
		}
		\bigg]
		\\
		&\quad\times
		\mathbb E\bigg[
		\left|
		\sigma\bigl(u(t_{i-1}(\delta),x)\bigr)
		\right|^\gamma
		\left|
		\sigma\bigl(u(t_{j-1}(\delta),x)\bigr)
		\right|^\gamma
		\\
		&\hspace{19mm}\times
		\bigg(
		\left|
		\widetilde\Delta(t_{j-1},\delta)
		\right|^\gamma
		-
		\mu_\gamma C_{\alpha,\beta}^{\gamma/2}
		\delta^{
			\frac{\gamma}{2}
			\left(1-\frac{\beta}{\alpha}\right)
		}
		\bigg)
		\bigg].
	\end{align*}
	
	Since
	\(\widetilde\Delta(t_{i-1},\delta)\) is centered Gaussian, it follows
	that
	\begin{align}
		&\mathbb E
		\left|
		\widetilde\Delta(t_{i-1},\delta)
		\right|^\gamma
		\notag\\
		&=
		\mu_\gamma
		\left(
		\mathbb E
		\widetilde\Delta^2(t_{i-1},\delta)
		\right)^{\gamma/2}
		\notag\\
		&=
		\mu_\gamma
		C_{\alpha,\beta}^{\gamma/2}
		\delta^{
			\frac{\gamma}{2}
			\left(1-\frac{\beta}{\alpha}\right)
		}
		+
		O\left(
		\delta^{
			2+
			\frac{\gamma-2}{2}
			\left(1-\frac{\beta}{\alpha}\right)
		}
		\right).
		\label{eq:gamma-additive-moment-expansion}
	\end{align}
	Hence,
	\[
	\left|\mathbb E
	\left|
	\widetilde\Delta(t_{i-1},\delta)
	\right|^\gamma
	-
	\mu_\gamma C_{\alpha,\beta}^{\gamma/2}
	\delta^{
		\frac{\gamma}{2}
		\left(1-\frac{\beta}{\alpha}\right)
	}
	\right|\leq C
	\delta^{
		2+
		\frac{\gamma-2}{2}
		\left(1-\frac{\beta}{\alpha}\right)
	}.
	\]
	On the other hand, by H\"older's inequality, the uniform moment
	bounds of \(u\), and
	\eqref{eq:gamma-centered-fourth-moment}, we obtain
	\begin{align}
		&\Bigg|
		\mathbb E\bigg[
		\left|
		\sigma\bigl(u(t_{i-1}(\delta),x)\bigr)
		\right|^\gamma
		\left|
		\sigma\bigl(u(t_{j-1}(\delta),x)\bigr)
		\right|^\gamma
		\notag\\
		&\hspace{25mm}\times
		\bigg(
		\left|
		\widetilde\Delta(t_{j-1},\delta)
		\right|^\gamma
		-
		\mu_\gamma C_{\alpha,\beta}^{\gamma/2}
		\delta^{
			\frac{\gamma}{2}
			\left(1-\frac{\beta}{\alpha}\right)
		}
		\bigg)
		\bigg]
		\Bigg|
		\notag\\
		&\leq
		\left(
		\mathbb E
		\left|
		\sigma\bigl(u(t_{i-1}(\delta),x)\bigr)
		\right|^{4\gamma}
		\right)^{1/4}
		\left(
		\mathbb E
		\left|
		\sigma\bigl(u(t_{j-1}(\delta),x)\bigr)
		\right|^{4\gamma}
		\right)^{1/4}
		\notag\\
		&\qquad\times
		\left\|
		\left|
		\widetilde\Delta(t_{j-1},\delta)
		\right|^\gamma
		-
		\mu_\gamma C_{\alpha,\beta}^{\gamma/2}
		\delta^{
			\frac{\gamma}{2}
			\left(1-\frac{\beta}{\alpha}\right)
		}
		\right\|_2
		\notag\\
		&\leq
		C
		\left\|
		\left|
		\widetilde\Delta(t_{j-1},\delta)
		\right|^\gamma
		-
		\mu_\gamma C_{\alpha,\beta}^{\gamma/2}
		\delta^{
			\frac{\gamma}{2}
			\left(1-\frac{\beta}{\alpha}\right)
		}
		\right\|_4
		\notag\\
		&\leq
		C
		\delta^{
			\frac{\gamma}{2}
			\left(1-\frac{\beta}{\alpha}\right)
		}.
		\label{eq:gamma-far-second-factor}
	\end{align}
%
%
Consequently,
	\begin{align}
		|J_{2,2}|
		&\leq
		C
		n^{
			\gamma\left(1-\frac{\beta}{\alpha}\right)-2
		}
		n^2
		\delta^{
			2+
			(\gamma-1)
			\left(1-\frac{\beta}{\alpha}\right)
		}
		\notag\\
		&=
		C
		n^{
			\gamma\left(1-\frac{\beta}{\alpha}\right)
		}
		\delta^{
			2+
			(\gamma-1)
			\left(1-\frac{\beta}{\alpha}\right)
		}
		\notag\\
		&\leq
		C\delta^{1+\beta/\alpha}.
		\label{eq:gamma-J2-far-final}
	\end{align}
	
	Combining \eqref{eq:gamma-J2-near-final} and
	\eqref{eq:gamma-J2-far-final} yields
	\begin{equation}\label{eq:gamma-J2-final}
			\mathbb E|J_2|^2
		\leq
		|J_{2,1}|+|J_{2,2}|
		\leq
		C\left(
		\delta^\kappa+\delta^{1+\beta/\alpha}
		\right).
	\end{equation}
	
	\medskip
	\noindent
	\textit{Estimate of \(J_3\).}
	
	We first rewrite it as 
	\begin{equation*}
		J_3
		=
		\mu_\gamma C_{\alpha,\beta}^{\gamma/2}
		(T_2-T_1)^{
			\frac{\gamma}{2}
			\left(1-\frac{\beta}{\alpha}\right)-1
		}
		\sum_{i=1}^n
		\int_{t_{i-1}}^{t_i}
		\bigg[
		\left|
		\sigma\bigl(u(t_{i-1}(\delta),x)\bigr)
		\right|^\gamma
		-
		|\sigma(u(r,x))|^\gamma
		\bigg]
		\,\mathrm dr.
	\end{equation*}
	
	By the Lipschitz continuity of \(\sigma\), the
	uniform moment bounds for \(u\), and the inequality
	\[
	\left||a|^\gamma-|b|^\gamma\right|
	\leq
	C_\gamma|a-b|
	\left(
	|a|^{\gamma-1}+|b|^{\gamma-1}
	\right),
	\]
	we obtain
	\begin{equation*}
		\mathbb E\left|
		\left|
		\sigma\bigl(u(t_{i-1}(\delta),x)\bigr)
		\right|^\gamma
		-
		|\sigma(u(r,x))|^\gamma
		\right|\leq
		C
		\left(
		\mathbb E
		\left|
		u(t_{i-1}(\delta),x)-u(r,x)
		\right|^2
		\right)^{1/2}.
	\end{equation*}
	For \(r\in[t_{i-1},t_i]\),
	\[
	0\leq
	r-(t_{i-1}(\delta))
	\leq
	\delta+\delta^\kappa
	\leq
	2\delta^\kappa.
	\]
	Therefore, the temporal increment estimate gives
	\[
	\mathbb E\left|
	\left|
	\sigma\bigl(u(t_{i-1}(\delta),x)\bigr)
	\right|^\gamma
	-
	|\sigma(u(r,x))|^\gamma
	\right|
	\leq
	C
	\delta^{
		\frac{\kappa}{2}
		\left(1-\frac{\beta}{\alpha}\right)
	}.
	\]
	It follows that
	\begin{align}
		\mathbb E|J_3|
		&\leq
		C
		\sum_{i=1}^n
		\int_{t_{i-1}}^{t_i}
		\delta^{
			\frac{\kappa}{2}
			\left(1-\frac{\beta}{\alpha}\right)
		}
		\,\mathrm dr
		\notag\\
		&\leq
		C
		\delta^{
			\frac{\kappa}{2}
			\left(1-\frac{\beta}{\alpha}\right)
		}
		\longrightarrow0,\qquad n\to \infty.
		\label{eq:gamma-J3-final}
	\end{align}
	
	We finally collect the preceding estimates. Fix
	\[
	0<\eta<
	\frac{
		\beta\kappa_*(\alpha+\beta)
	}{
		2\alpha(2\alpha+\beta\kappa_*)
	}.
	\]
	Since
	\begin{equation*}
		\frac12
		\left[
		\frac{
			2[\alpha+\beta(\kappa_*-1)]
		}{
			2\alpha+\beta\kappa_*
		}
		-
		\left(1-\frac{\beta}{\alpha}\right)
		\right]=
		\frac{
			\beta\kappa_*(\alpha+\beta)
		}{
			2\alpha(2\alpha+\beta\kappa_*)
		},
	\end{equation*}
	we may choose \(\varepsilon>0\) sufficiently small so that
	\[
	\frac12
	\left[
	\frac{
		2[\alpha+\beta(\kappa_*-1)]
	}{
		2\alpha+\beta\kappa_*
	}
	-
	\left(1-\frac{\beta}{\alpha}\right)
	-
	\varepsilon
	\right]
	\geq\eta.
	\]
	It then follows from \eqref{eq:gamma-J1-final} that
	\begin{equation}
		\label{eq:gamma-J1-quantitative}
		\mathbb E|J_1|
		\leq
		C_{\alpha,\beta,\gamma,T_1,T_2,\eta}\,
		\delta^\eta.
	\end{equation}
	
	By \eqref{eq:gamma-J2-final} and the Cauchy--Schwarz inequality,
	\begin{align}
		\mathbb E|J_2|
		&\leq
		\left(\mathbb E|J_2|^2\right)^{1/2}
		\notag\\
		&\leq
		C_{\alpha,\beta,\gamma,T_1,T_2}
		\left(
		\delta^{\kappa/2}
		+
		\delta^{\frac12(1+\beta/\alpha)}
		\right).
		\label{eq:gamma-J2-quantitative}
	\end{align}
	Moreover, \eqref{eq:gamma-J3-final} gives
	\begin{equation}
		\label{eq:gamma-J3-quantitative}
		\mathbb E|J_3|
		\leq
		C_{\alpha,\beta,\gamma,T_1,T_2}\,
		\delta^{
			\frac{\kappa}{2}
			\left(1-\frac{\beta}{\alpha}\right)
		}.
	\end{equation}
	
	Since
	\[
	\kappa_*
	\leq
	\frac{\alpha-\beta}{\beta}
	\qquad\text{and}\qquad
	\kappa
	=
	\frac{2\alpha}{2\alpha+\beta\kappa_*},
	\]
	we have
	\begin{align*}
		\frac{
			\beta\kappa_*(\alpha+\beta)
		}{
			2\alpha(2\alpha+\beta\kappa_*)
		}
		&\leq
		\frac{\kappa}{2},
		\\
		\frac{
			\beta\kappa_*(\alpha+\beta)
		}{
			2\alpha(2\alpha+\beta\kappa_*)
		}
		&\leq
		\frac12
		\left(1+\frac{\beta}{\alpha}\right),
		\\
		\frac{
			\beta\kappa_*(\alpha+\beta)
		}{
			2\alpha(2\alpha+\beta\kappa_*)
		}
		&\leq
		\frac{\kappa}{2}
		\left(1-\frac{\beta}{\alpha}\right).
	\end{align*}
	Therefore, the powers appearing in
	\eqref{eq:gamma-J2-quantitative} and
	\eqref{eq:gamma-J3-quantitative} are not smaller than \(\eta\).
	Since \(0<\delta\leq1\), combining
	\eqref{eq:gamma-J1-quantitative}--
	\eqref{eq:gamma-J3-quantitative} yields
	\[
	\mathbb E|J_1+J_2+J_3|
	\leq
	C_{\alpha,\beta,\gamma,T_1,T_2,\eta}
	\delta^\eta.
	\]
	All the preceding estimates are uniform with respect to
	\(x\in\mathbb R\). Hence,
	\begin{align*}
		&\sup_{x\in\mathbb R}
		\mathbb E\bigg|
		n^{
			\frac{\gamma}{2}
			\left(1-\frac{\beta}{\alpha}\right)-1
		}
		\sum_{i=1}^n
		|u(t_i,x)-u(t_{i-1},x)|^\gamma
		\\
		&\qquad-
		\mu_\gamma C_{\alpha,\beta}^{\gamma/2}
		(T_2-T_1)^{
			\frac{\gamma}{2}
			\left(1-\frac{\beta}{\alpha}\right)-1
		}
		\int_{T_1}^{T_2}
		|\sigma(u(r,x))|^\gamma
		\,\mathrm dr
		\bigg|
		\\
		&\qquad\leq
		C_{\alpha,\beta,\gamma,T_1,T_2,\eta}
		\delta^\eta
		\leq
		C_{\alpha,\beta,\gamma,T_1,T_2,\eta}
		n^{-\eta}.
	\end{align*}
	This proves both
	\eqref{eq:normalized-temporal-gamma-variation} and
	\eqref{eq:normalized-temporal-gamma-rate}.
\end{proof}

\begin{corollary}
	\label{cor:normalized-temporal-quadratic-variation}
	Let \(0<T_1<T_2<\infty\) and \(x\in\mathbb R\) be fixed. Set
	\[
	\delta:=\frac{T_2-T_1}{n},
	\qquad
	t_i:=T_1+i\delta,
	\qquad i=0,1,\ldots,n.
	\]
	Then
	\begin{equation}
		\label{eq:normalized-temporal-quadratic-variation}
		\frac{1}{n^{\beta/\alpha}}
		\sum_{i=1}^n
		\left[
		u(t_i,x)-u(t_{i-1},x)
		\right]^2
		\overset{L^1(\Omega)}{\longrightarrow}
		C_{\alpha,\beta}
		(T_2-T_1)^{-\beta/\alpha}
		\int_{T_1}^{T_2}
		\sigma^2(u(r,x))
		\,\mathrm dr.
	\end{equation}
\end{corollary}

\begin{proof}
	Taking \(\gamma=2\) in Theorem
	\ref{thm:normalized-temporal-gamma-variation}, and using
	\[
	\mu_2=1,
	\qquad
	\frac{\gamma}{2}
	\left(1-\frac{\beta}{\alpha}\right)-1
	=
	-\frac{\beta}{\alpha},
	\]
	gives the desired result.
\end{proof}

\begin{corollary}
	\label{cor:critical-temporal-variation}
	Let \(0<T_1<T_2<\infty\) and \(x\in\mathbb R\) be fixed. Set
	\[
	\delta:=\frac{T_2-T_1}{n},
	\qquad
	t_i:=T_1+i\delta,
	\qquad i=0,1,\ldots,n.
	\]
	Then
	\begin{equation}\label{eq:critical-temporal-variation}
		\sum_{i=1}^n
		\left|
		u(t_i,x)-u(t_{i-1},x)
		\right|^{\frac{2\alpha}{\alpha-\beta}}
		\overset{L^1(\Omega)}{\longrightarrow}
		\mu_{\frac{2\alpha}{\alpha-\beta}}
		C_{\alpha,\beta}^{\frac{\alpha}{\alpha-\beta}}
		\int_{T_1}^{T_2}
		\left|
		\sigma(u(r,x))
		\right|^{\frac{2\alpha}{\alpha-\beta}}
		\,\mathrm dr.
	\end{equation}
\end{corollary}

\begin{proof}
	Taking
	\[
	\gamma=\frac{2\alpha}{\alpha-\beta}
	\]
	in Theorem \ref{thm:normalized-temporal-gamma-variation}, we have
	\[
	\frac{\gamma}{2}
	\left(1-\frac{\beta}{\alpha}\right)=1.
	\]
	Consequently, both the normalizing power of \(n\) and the power of
	\(T_2-T_1\) in
	\eqref{eq:normalized-temporal-gamma-variation} are equal to zero.
	This gives \eqref{eq:critical-temporal-variation}.
\end{proof}

\subsection{Spatially averaged temporal power variations}
\label{subsec:spatially-averaged-temporal-power-variations}

Motivated by Hildebrandt and Trabs~\cite{Hildebrandt2021}, we next
average the normalized temporal \(\gamma\)-variations over spatial
locations. The uniform estimate
\eqref{eq:normalized-temporal-gamma-rate} allows the numbers of
temporal and spatial observations to tend to infinity jointly.

\begin{proposition}
	\label{pro:space-averaged-temporal-gamma-variation}
	Let \(0<T_1<T_2<\infty\), and set
	\[
	\delta:=\frac{T_2-T_1}{n},
	\qquad
	t_i:=T_1+i\delta,
	\qquad i=0,1,\ldots,n.
	\]
	For \(m,n\geq1\) and every fixed \(\gamma\geq2\), define
	\begin{align}
		V_{n,m}^{(\gamma)}(u)
		:={}&
		\frac1m
		\sum_{j=0}^{m-1}
		n^{
			\frac{\gamma}{2}
			\left(1-\frac{\beta}{\alpha}\right)-1
		}
		\sum_{i=1}^n
		\left|
		u(t_i,x_j)-u(t_{i-1},x_j)
		\right|^\gamma,
		\label{eq:space-averaged-gamma-variation}
	\end{align}
	where
	\[
	x_j:=\frac jm,
	\qquad j=0,1,\ldots,m.
	\]
	Then, as \(n,m\to\infty\),
	\begin{equation}\label{eq:space-averaged-gamma-limit}
		V_{n,m}^{(\gamma)}(u)
		\overset{L^1(\Omega)}{\longrightarrow}
		\mu_\gamma C_{\alpha,\beta}^{\gamma/2}
		(T_2-T_1)^{
			\frac{\gamma}{2}
			\left(1-\frac{\beta}{\alpha}\right)-1
		}
		\int_0^1
		\int_{T_1}^{T_2}
		|\sigma(u(t,x))|^\gamma
		\,\mathrm dt\,\mathrm dx.
	\end{equation}
	
	More precisely, for every
	\[
	0<\eta<
	\frac{
		\beta\kappa_*(\alpha+\beta)
	}{
		2\alpha(2\alpha+\beta\kappa_*)
	}
	\]
	and every
	\[
	0<\zeta<\frac{\kappa_*}{2},
	\]
	there exists a constant
	\(C_{\eta,\zeta,T_1,T_2,\gamma}<\infty\) such that, for all
	sufficiently large \(n\) and \(m\),
	\begin{align}
		&\mathbb E\bigg|
		V_{n,m}^{(\gamma)}(u)
		-
		\mu_\gamma C_{\alpha,\beta}^{\gamma/2}
		(T_2-T_1)^{
			\frac{\gamma}{2}
			\left(1-\frac{\beta}{\alpha}\right)-1
		}
		\notag\\
		&\hspace{35mm}\times
		\int_0^1
		\int_{T_1}^{T_2}
		|\sigma(u(t,x))|^\gamma
		\,\mathrm dt\,\mathrm dx
		\bigg|
		\notag\\
		&\qquad\leq
		C_{\eta,\zeta,T_1,T_2,\gamma}
		\left(
		n^{-\eta}+m^{-\zeta}
		\right).
		\label{eq:space-averaged-gamma-rate}
	\end{align}
\end{proposition}
\begin{proof}
	We decompose the difference between $V_{n,m}^{(\gamma)}(u)$ and its limit in
	\eqref{eq:space-averaged-gamma-limit} as
	\[
	D_1+D_2,
	\]
	where
	\begin{equation*}
		D_1
		:=
		V_{n,m}^{(\gamma)}(u)-
		\frac1m
		\sum_{j=0}^{m-1}
		\mu_\gamma C_{\alpha,\beta}^{\gamma/2}
		(T_2-T_1)^{
			\frac{\gamma}{2}
			\left(1-\frac{\beta}{\alpha}\right)-1
		}
		\int_{T_1}^{T_2}
		|\sigma(u(t,x_j))|^\gamma
		\,\mathrm dt
	\end{equation*}
	and
	\begin{align*}
		D_2
		:={}&
		\frac1m
		\sum_{j=0}^{m-1}
		\mu_\gamma C_{\alpha,\beta}^{\gamma/2}
		(T_2-T_1)^{
			\frac{\gamma}{2}
			\left(1-\frac{\beta}{\alpha}\right)-1
		}
		\int_{T_1}^{T_2}
		|\sigma(u(t,x_j))|^\gamma
		\,\mathrm dt
		\\
		&-
		\mu_\gamma C_{\alpha,\beta}^{\gamma/2}
		(T_2-T_1)^{
			\frac{\gamma}{2}
			\left(1-\frac{\beta}{\alpha}\right)-1
		}
		\int_0^1
		\int_{T_1}^{T_2}
		|\sigma(u(t,x))|^\gamma
		\,\mathrm dt\,\mathrm dx.
	\end{align*}
	
	For \(\eta\) satisfying \eqref{eq:choice-temporal-rate}, the triangle
	inequality and \eqref{eq:normalized-temporal-gamma-rate} yield
	\begin{align}
		\mathbb E|D_1|
		&\leq
		\frac1m
		\sum_{j=0}^{m-1}
		\mathbb E\bigg|
		n^{
			\frac{\gamma}{2}
			\left(1-\frac{\beta}{\alpha}\right)-1
		}
		\sum_{i=1}^n
		|u(t_i,x_j)-u(t_{i-1},x_j)|^\gamma
		\notag\\
		&\qquad-
		\mu_\gamma C_{\alpha,\beta}^{\gamma/2}
		(T_2-T_1)^{
			\frac{\gamma}{2}
			\left(1-\frac{\beta}{\alpha}\right)-1
		}
		\int_{T_1}^{T_2}
		|\sigma(u(t,x_j))|^\gamma
		\,\mathrm dt
		\bigg|
		\notag\\
		&\leq
		C_{\alpha,\beta,\gamma,T_1,T_2,\eta}
		n^{-\eta}.
		\label{eq:space-average-D1}
	\end{align}
	
	We next estimate \(D_2\). Since
	\[
	\frac1m
	\sum_{j=0}^{m-1}
	\int_{T_1}^{T_2}
	|\sigma(u(t,x_j))|^\gamma
	\,\mathrm dt
	=
	\sum_{j=0}^{m-1}
	\int_{x_j}^{x_{j+1}}
	\int_{T_1}^{T_2}
	|\sigma(u(t,x_j))|^\gamma
	\,\mathrm dt\,\mathrm dx,
	\]
	we obtain
	\begin{equation*}
		\mathbb E|D_2|
		\leq
		C_{\alpha,\beta,\gamma,T_1,T_2}
		\sum_{j=0}^{m-1}
		\int_{x_j}^{x_{j+1}}
		\int_{T_1}^{T_2}
		\mathbb E\left[\left|
		|\sigma(u(t,x_j))|^\gamma
		-
		|\sigma(u(t,x))|^\gamma
		\right|\right]
		\,\mathrm dt\,\mathrm dx.
	\end{equation*}
	
	Using
	\[
	\left|
	|a|^\gamma-|b|^\gamma
	\right|
	\leq
	C_\gamma
	|a-b|
	\left(
	|a|^{\gamma-1}+|b|^{\gamma-1}
	\right),
	\]
	the Lipschitz continuity of \(\sigma\), the uniform
	moment bounds for \(u\), and the Cauchy--Schwarz inequality, we have
	\begin{equation*}
		\mathbb E\left|
		|\sigma(u(t,x_j))|^\gamma
		-
		|\sigma(u(t,x))|^\gamma
		\right|\leq
		C_{\gamma,T_2}
		\left(
		\mathbb E
		|u(t,x_j)-u(t,x)|^2
		\right)^{1/2}.
	\end{equation*}
	For every
	\[
	0<\zeta<\frac{\kappa_*}{2},
	\]
	the spatial increment estimate can be applied with exponent
	\(2\zeta<\kappa_*\), giving
	\[
	\mathbb E
	|u(t,x_j)-u(t,x)|^2
	\leq
	C_{\zeta,T_2}|x_j-x|^{2\zeta}.
	\]
	Consequently,
	\[
	\mathbb E\left|
	|\sigma(u(t,x_j))|^\gamma
	-
	|\sigma(u(t,x))|^\gamma
	\right|
	\leq
	C_{\zeta,\gamma,T_2}
	|x_j-x|^\zeta.
	\]
	It follows that
	\begin{align}
		\mathbb E|D_2|
		&\leq
		C_{\zeta,\gamma,T_1,T_2}
		\sum_{j=0}^{m-1}
		\int_{x_j}^{x_{j+1}}
		|x_j-x|^\zeta
		\,\mathrm dx
		\notag\\
		&\leq
		C_{\zeta,\gamma,T_1,T_2}
		m\left(\frac1m\right)^{1+\zeta}
		\notag\\
		&=
		C_{\zeta,\gamma,T_1,T_2}
		m^{-\zeta}.
		\label{eq:space-average-D2}
	\end{align}
	
	Combining \eqref{eq:space-average-D1} and
	\eqref{eq:space-average-D2}, we obtain
	\eqref{eq:space-averaged-gamma-rate}. The convergence in
	\eqref{eq:space-averaged-gamma-limit} follows immediately.
\end{proof}

\begin{remark}
	\label{rem:space-averaged-special-cases}
	Proposition \ref{pro:space-averaged-temporal-gamma-variation}
	contains the spatially averaged normalized temporal quadratic variation and the
	spatially averaged higher order temporal variation as special cases.
	
	When \(\gamma=2\),
	\begin{align*}
		&\frac1m
		\sum_{j=0}^{m-1}
		\frac1{n^{\beta/\alpha}}
		\sum_{i=1}^n
		\left[
		u(t_i,x_j)-u(t_{i-1},x_j)
		\right]^2
		\\
		&\quad\overset{L^1(\Omega)}{\longrightarrow}
		C_{\alpha,\beta}
		(T_2-T_1)^{-\beta/\alpha}
		\int_0^1
		\int_{T_1}^{T_2}
		\sigma^2(u(t,x))
		\,\mathrm dt\,\mathrm dx.
	\end{align*}
	
	When
	\[
	\gamma=\frac{2\alpha}{\alpha-\beta},
	\]
	the temporal normalization factor equals one, and
	\begin{align*}
		&\frac1m
		\sum_{j=0}^{m-1}
		\sum_{i=1}^n
		\left|
		u(t_i,x_j)-u(t_{i-1},x_j)
		\right|^{\frac{2\alpha}{\alpha-\beta}}
		\\
		&\quad\overset{L^1(\Omega)}{\longrightarrow}
		\mu_{\frac{2\alpha}{\alpha-\beta}}
		C_{\alpha,\beta}^{\frac{\alpha}{\alpha-\beta}}
		\int_0^1
		\int_{T_1}^{T_2}
		\left|
		\sigma(u(t,x))
		\right|^{\frac{2\alpha}{\alpha-\beta}}
		\,\mathrm dt\,\mathrm dx.
	\end{align*}
\end{remark}

\section{Applications to parameter estimation}
\label{sec:applications}

In this section, we use the temporal variation results established
above to construct consistent estimators for the drift parameter
\(\nu>0\). We subsequently show that the same high-frequency
observations also yield a consistent estimator for the ratio
\(\beta/\alpha\).

Throughout this section, the coefficient \(\sigma\) is assumed to be
known. For the estimation of \(\nu\) in Subsection~\ref{sebsec:4.1}, the fractional
orders \(\alpha\) and \(\beta\) are also assumed to be known. The solution is observed over a fixed time interval
\([T_1,T_2]\), where
\[
0<T_1<T_2<\infty.
\]
For each \(n\geq1\), let
\[
\delta:=\frac{T_2-T_1}{n},
\qquad
t_i:=T_1+i\delta,
\qquad i=0,1,\ldots,n.
\]

We first restore the dependence on \(\nu\), which was set equal to one
in the preceding section. Let \(u_\nu\) denote the mild solution
corresponding to the kernel \(G_{\alpha,\beta}^{(\nu)}\). Its Fourier
transform is given by
\[
\mathcal F
G_{\alpha,\beta}^{(\nu)}(t,\cdot)(\xi)
=
E_\beta\left(
-\nu t^\beta|\xi|^\alpha
\right).
\]
A change of variables in the Fourier inversion formula gives
\begin{equation}
	\label{eq:kernel-nu-scaling}
	G_{\alpha,\beta}^{(\nu)}(t,x)
	=
	\nu^{-1/\alpha}
	G_{\alpha,\beta}^{(1)}
	\left(
	t,\nu^{-1/\alpha}x
	\right).
\end{equation}
Consequently,
\begin{align}
	&\int_{\mathbb R}
	\left|
	G_{\alpha,\beta}^{(\nu)}(t+\delta,y)
	-
	G_{\alpha,\beta}^{(\nu)}(t,y)
	\right|^2
	\,\mathrm dy
	\notag\\
	&\qquad=
	\nu^{-1/\alpha}
	\int_{\mathbb R}
	\left|
	G_{\alpha,\beta}^{(1)}(t+\delta,y)
	-
	G_{\alpha,\beta}^{(1)}(t,y)
	\right|^2
	\,\mathrm dy.
	\label{eq:kernel-nu-temporal-difference}
\end{align}
Thus the leading constant in the second moment of an additive temporal
increment is
\[
C_{\alpha,\beta,\nu}
=
\nu^{-1/\alpha}C_{\alpha,\beta},
\]
where \(C_{\alpha,\beta}\) denotes the constant corresponding to
\(\nu=1\).

Repeating the argument in Theorem
\ref{thm:normalized-temporal-gamma-variation}, with
\(C_{\alpha,\beta}\) replaced by
\(\nu^{-1/\alpha}C_{\alpha,\beta}\), yields, for every fixed
\(\gamma\geq2\),
\begin{align}
	&n^{
		\frac{\gamma}{2}
		\left(1-\frac{\beta}{\alpha}\right)-1
	}
	\sum_{i=1}^n
	\left|
	u_\nu(t_i,x)-u_\nu(t_{i-1},x)
	\right|^\gamma
	\notag\\
	&\quad\xrightarrow{L^1(\Omega)}
	\mu_\gamma C_{\alpha,\beta}^{\gamma/2}
	\nu^{-\gamma/(2\alpha)}
	(T_2-T_1)^{
		\frac{\gamma}{2}
		\left(1-\frac{\beta}{\alpha}\right)-1
	}
	\int_{T_1}^{T_2}
	|\sigma(u_\nu(r,x))|^\gamma
	\,\mathrm dr.
	\label{eq:temporal-gamma-variation-general-nu}
\end{align}

Similarly, the spatially averaged variation satisfies
\begin{equation}\label{eq:space-averaged-gamma-variation-general-nu}
	V_{n,m}^{(\gamma)}(u_\nu)
	\xrightarrow{L^1(\Omega)}
	\mu_\gamma C_{\alpha,\beta}^{\gamma/2}
	\nu^{-\gamma/(2\alpha)}
	(T_2-T_1)^{
		\frac{\gamma}{2}
		\left(1-\frac{\beta}{\alpha}\right)-1
	}
	\int_0^1
	\int_{T_1}^{T_2}
	|\sigma(u_\nu(t,x))|^\gamma
	\,\mathrm dt\,\mathrm dx.
\end{equation}

\subsection{Estimation of the drift parameter}\label{sebsec:4.1}

We first construct an estimator from temporal observations at a fixed
spatial point.

\begin{proposition}
	\label{prop:nu-estimator-temporal}
	Let \(x\in\mathbb R\) and \(\gamma\geq2\) be fixed.
	Define
	\begin{align}
		V_n^{(\gamma)}(x)
		:={}&
		n^{
			\frac{\gamma}{2}
			\left(1-\frac{\beta}{\alpha}\right)-1
		}
		\sum_{i=1}^n
		\left|
		u_\nu(t_i,x)-u_\nu(t_{i-1},x)
		\right|^\gamma,
		\label{eq:temporal-gamma-statistic-nu}
	\end{align}
	and
	\begin{equation}
		\label{eq:temporal-sigma-riemann-sum}
		S_n^{(\gamma)}(x)
		:=
		\delta
		\sum_{i=1}^n
		\left|
		\sigma(u_\nu(t_{i-1},x))
		\right|^\gamma.
	\end{equation}
	Assume that
	\begin{equation}
		\label{eq:temporal-nondegeneracy}
		\mathbb P\left(
		\int_{T_1}^{T_2}
		|\sigma(u_\nu(r,x))|^\gamma
		\,\mathrm dr
		>0
		\right)
		=1.
	\end{equation}
	On the event \(\{V_n^{(\gamma)}(x)>0\}\), define
	\begin{equation}
		\label{eq:nu-estimator-temporal}
		\widehat\nu_{n,\gamma}(x)
		:=
		\left[
		\frac{
			\mu_\gamma
			C_{\alpha,\beta}^{\gamma/2}
			(T_2-T_1)^{
				\frac{\gamma}{2}
				\left(1-\frac{\beta}{\alpha}\right)-1
			}
			S_n^{(\gamma)}(x)
		}{
			V_n^{(\gamma)}(x)
		}
		\right]^{2\alpha/\gamma},
	\end{equation}
	and assign it any fixed positive value on the complementary event.
	Then
	\[
	\widehat\nu_{n,\gamma}(x)
	\xrightarrow{\mathbb P}
	\nu,
	\qquad n\to\infty.
	\]
\end{proposition}

\begin{proof}
	By the continuity of \(u_\nu(\cdot,x)\) and the Lipschitz
	continuity of \(\sigma\), the Riemann sums satisfy
	\begin{equation}
		\label{eq:temporal-sigma-riemann-convergence}
		S_n^{(\gamma)}(x)
		\longrightarrow
		\int_{T_1}^{T_2}
		|\sigma(u_\nu(r,x))|^\gamma
		\,\mathrm dr
	\end{equation}
	almost surely.
	
	On the other hand, \eqref{eq:temporal-gamma-variation-general-nu}
	gives
	\begin{align*}
		V_n^{(\gamma)}(x)
		\xrightarrow{\mathbb P}{}
		\mu_\gamma C_{\alpha,\beta}^{\gamma/2}
		\nu^{-\gamma/(2\alpha)}
		(T_2-T_1)^{
			\frac{\gamma}{2}
			\left(1-\frac{\beta}{\alpha}\right)-1
		}
		\int_{T_1}^{T_2}
		|\sigma(u_\nu(r,x))|^\gamma
		\,\mathrm dr.
	\end{align*}
	Under \eqref{eq:temporal-nondegeneracy}, the limiting random
	variable is strictly positive almost surely. Hence,
	\[
	\mathbb P\left(
	V_n^{(\gamma)}(x)>0
	\right)
	\longrightarrow1.
	\]
	Combining the preceding two convergences and applying the continuous
	mapping theorem, we obtain
	\begin{align*}
		\frac{
			\mu_\gamma
			C_{\alpha,\beta}^{\gamma/2}
			(T_2-T_1)^{
				\frac{\gamma}{2}
				\left(1-\frac{\beta}{\alpha}\right)-1
			}
			S_n^{(\gamma)}(x)
		}{
			V_n^{(\gamma)}(x)
		}\xrightarrow{\mathbb P}
		\nu^{\gamma/(2\alpha)}.
	\end{align*}
	Raising both sides to the power \(2\alpha/\gamma\) proves
	\[
	\widehat\nu_{n,\gamma}(x)
	\xrightarrow{\mathbb P}
	\nu.
	\]
\end{proof}

\begin{remark}
	\label{rem:nu-estimator-temporal-special-cases}
	Two important special cases of Proposition
	\ref{prop:nu-estimator-temporal} are worth recording.
	
	When \(\gamma=2\), we have \(\mu_2=1\), and
	\begin{equation}
		\label{eq:nu-estimator-quadratic}
		\widehat\nu_{n,2}(x)
		=
		\left[
		\frac{
			C_{\alpha,\beta}
			(T_2-T_1)^{-\beta/\alpha}
			\delta
			\displaystyle\sum_{i=1}^n
			\sigma^2(u_\nu(t_{i-1},x))
		}{
			n^{-\beta/\alpha}
			\displaystyle\sum_{i=1}^n
			\left[
			u_\nu(t_i,x)-u_\nu(t_{i-1},x)
			\right]^2
		}
		\right]^\alpha.
	\end{equation}
	
	For the critical exponent
	\[
	p_{\alpha,\beta}
	:=
	\frac{2\alpha}{\alpha-\beta},
	\]
	the normalization involving \(n\) and the power of \(T_2-T_1\)
	both disappear. Therefore,
	\begin{equation}
		\label{eq:nu-estimator-critical}
		\widehat\nu_{n,p_{\alpha,\beta}}(x)
		=
		\left[
		\frac{
			\mu_{p_{\alpha,\beta}}
			C_{\alpha,\beta}^{p_{\alpha,\beta}/2}
			\delta
			\displaystyle\sum_{i=1}^n
			\left|
			\sigma(u_\nu(t_{i-1},x))
			\right|^{p_{\alpha,\beta}}
		}{
			\displaystyle\sum_{i=1}^n
			\left|
			u_\nu(t_i,x)-u_\nu(t_{i-1},x)
			\right|^{p_{\alpha,\beta}}
		}
		\right]^{\alpha-\beta}.
	\end{equation}
	Both estimators converge to \(\nu\) in probability.
	Here
	\[
	\mu_{p_{\alpha,\beta}}
	=
	\mathbb E|Z|^{p_{\alpha,\beta}},
	\qquad Z\sim N(0,1).
	\]
\end{remark}

We next use temporal observations collected over an increasing number
of spatial locations.

\begin{proposition}
	\label{prop:nu-estimator-space-averaged}
	Let \(V_{n,m}^{(\gamma)}(u_\nu)\) be defined as in
	\eqref{eq:space-averaged-gamma-variation}, with \(u\) replaced by
	\(u_\nu\), and define
	\begin{equation}
		\label{eq:space-time-sigma-riemann-sum}
		S_{n,m}^{(\gamma)}
		:=
		\frac{\delta}{m}
		\sum_{j=0}^{m-1}
		\sum_{i=1}^n
		\left|
		\sigma(u_\nu(t_{i-1},x_j))
		\right|^\gamma,
		\qquad
		x_j:=\frac jm.
	\end{equation}
	Assume that
	\begin{equation}
		\label{eq:space-time-nondegeneracy}
		\mathbb P\left(
		\int_0^1
		\int_{T_1}^{T_2}
		|\sigma(u_\nu(t,x))|^\gamma
		\,\mathrm dt\,\mathrm dx
		>0
		\right)
		=1.
	\end{equation}
	On the event
	\(\{V_{n,m}^{(\gamma)}(u_\nu)>0\}\), define
	\begin{equation}
		\label{eq:nu-estimator-space-averaged}
		\widehat\nu_{n,m,\gamma}
		:=
		\left[
		\frac{
			\mu_\gamma
			C_{\alpha,\beta}^{\gamma/2}
			(T_2-T_1)^{
				\frac{\gamma}{2}
				\left(1-\frac{\beta}{\alpha}\right)-1
			}
			S_{n,m}^{(\gamma)}
		}{
			V_{n,m}^{(\gamma)}(u_\nu)
		}
		\right]^{2\alpha/\gamma},
	\end{equation}
	and assign it any fixed positive value otherwise. Then, as
	\(n,m\to\infty\),
	\[
	\widehat\nu_{n,m,\gamma}
	\xrightarrow{\mathbb P}
	\nu.
	\]
	No relation between the rates at which \(n\) and \(m\) tend to
	infinity is required.
\end{proposition}

\begin{proof}
	The joint continuity of \(u_\nu\) and the Lipschitz continuity of
	\(\sigma\) imply that
	\begin{align}
		S_{n,m}^{(\gamma)}
		\longrightarrow
		\int_0^1
		\int_{T_1}^{T_2}
		|\sigma(u_\nu(t,x))|^\gamma
		\,\mathrm dt\,\mathrm dx
		\label{eq:space-time-sigma-riemann-convergence}
	\end{align}
	almost surely as \(n,m\to\infty\).
	
	Moreover, by
	\eqref{eq:space-averaged-gamma-variation-general-nu},
	\begin{align*}
		V_{n,m}^{(\gamma)}(u_\nu)
		\xrightarrow{\mathbb P}{}
		\mu_\gamma C_{\alpha,\beta}^{\gamma/2}
		\nu^{-\gamma/(2\alpha)}
		(T_2-T_1)^{
			\frac{\gamma}{2}
			\left(1-\frac{\beta}{\alpha}\right)-1
		}
		\int_0^1
		\int_{T_1}^{T_2}
		|\sigma(u_\nu(t,x))|^\gamma
		\,\mathrm dt\,\mathrm dx.
	\end{align*}
	The limiting random variable is strictly positive almost surely by
	\eqref{eq:space-time-nondegeneracy}. The continuous mapping theorem
	therefore gives
	\[
	\widehat\nu_{n,m,\gamma}
	\xrightarrow{\mathbb P}
	\nu.
	\]
\end{proof}

\begin{remark}
	\label{rem:nu-estimator-space-averaged-special-cases}
	For \(\gamma=2\), Proposition
	\ref{prop:nu-estimator-space-averaged} gives
	\begin{equation}
		\widehat\nu_{n,m,2}
		=
		\left[
		\frac{
			C_{\alpha,\beta}
			(T_2-T_1)^{-\beta/\alpha}
			S_{n,m}^{(2)}
		}{
			V_{n,m}^{(2)}(u_\nu)
		}
		\right]^\alpha
		\xrightarrow{\mathbb P}
		\nu.
	\end{equation}
	
	For
	\[
	p_{\alpha,\beta}
	=
	\frac{2\alpha}{\alpha-\beta},
	\]
	the corresponding critical-variation estimator is
	\begin{equation}
		\widehat\nu_{n,m,p_{\alpha,\beta}}
		=
		\left[
		\frac{
			\mu_{p_{\alpha,\beta}}
			C_{\alpha,\beta}^{p_{\alpha,\beta}/2}
			S_{n,m}^{(p_{\alpha,\beta})}
		}{
			V_{n,m}^{(p_{\alpha,\beta})}(u_\nu)
		}
		\right]^{\alpha-\beta}
		\xrightarrow{\mathbb P}
		\nu.
	\end{equation}
\end{remark}

\subsection{Estimation of the ratio
\texorpdfstring{\(\beta/\alpha\)}{beta/alpha}}

The temporal roughness depends on \(\alpha\) and \(\beta\) through the
ratio \(\beta/\alpha\). We therefore first construct a consistent
estimator for this ratio.

\begin{proposition}
	\label{prop:beta-alpha-ratio-estimator}
	Let \(x\in\mathbb R\) and \(\gamma\geq2\) be fixed, and assume that
	\begin{equation}
		\label{eq:ratio-nondegeneracy}
		\mathbb P\left(
		\int_{T_1}^{T_2}
		|\sigma(u_\nu(r,x))|^\gamma
		\,\mathrm dr
		>0
		\right)
		=1.
	\end{equation}
	For every \(n\geq2\), whenever
	\[
	\frac1n
	\sum_{i=1}^n
	\left|
	\sigma\bigl(u_\nu(t_{i-1},x)\bigr)
	\right|^\gamma
	>0,
	\]
	define
	\begin{equation}
		\label{eq:definition-A-n-gamma}
		A_{n,\gamma}(x)
		:=
		\frac{
			\sum_{i=1}^n
			\left|
			u_\nu(t_i,x)-u_\nu(t_{i-1},x)
			\right|^\gamma
		}{
			\frac1n
			\sum_{i=1}^n
			\left|
			\sigma\bigl(u_\nu(t_{i-1},x)\bigr)
			\right|^\gamma
		}.
	\end{equation}
	If the denominator in
	\eqref{eq:definition-A-n-gamma} vanishes, we set
	\(A_{n,\gamma}(x)=0\).
	
	On the event
	\(\{A_{n,\gamma}(x)>0\}\), define
	\begin{equation}
		\label{eq:beta-alpha-ratio-estimator}
		\widehat r_{n,\gamma}
		:=
		1-\frac{2}{\gamma}
		+
		\frac{
			2\log A_{n,\gamma}(x)
		}{
			\gamma\log n
		},
	\end{equation}
	and assign it an arbitrary fixed value otherwise. Then
	\[
	\widehat r_{n,\gamma}
	\xrightarrow{\mathbb P}
	\frac{\beta}{\alpha},
	\qquad n\to\infty.
	\]
\end{proposition}

\begin{proof}
	By the continuity of \(u_\nu(\cdot,x)\) and the Lipschitz
	continuity of \(\sigma\), the Riemann sums satisfy
	\begin{equation}
		\label{eq:ratio-Riemann-sum}
		\frac1n
		\sum_{i=1}^n
		\left|
		\sigma\bigl(u_\nu(t_{i-1},x)\bigr)
		\right|^\gamma
		\longrightarrow
		\frac1{T_2-T_1}
		\int_{T_1}^{T_2}
		|\sigma(u_\nu(r,x))|^\gamma
		\,\mathrm dr
	\end{equation}
	almost surely.
	
	By \eqref{eq:ratio-nondegeneracy}, the limit on the right-hand
	side of \eqref{eq:ratio-Riemann-sum} is strictly positive almost
	surely. Consequently, the denominator in
	\eqref{eq:definition-A-n-gamma} is positive for all sufficiently
	large \(n\), almost surely. Hence, the convention
	\(A_{n,\gamma}(x)=0\) is asymptotically irrelevant.
	
	Combining \eqref{eq:ratio-Riemann-sum} with
	\eqref{eq:temporal-gamma-variation-general-nu}, we obtain
	\begin{equation}
		\label{eq:A-n-gamma-limit}
		n^{
			\frac{\gamma}{2}
			\left(1-\frac{\beta}{\alpha}\right)-1
		}
		A_{n,\gamma}(x)
		\xrightarrow{\mathbb P}
		\mu_\gamma
		C_{\alpha,\beta}^{\gamma/2}
		\nu^{-\gamma/(2\alpha)}
		(T_2-T_1)^{
			\frac{\gamma}{2}
			\left(1-\frac{\beta}{\alpha}\right)
		}.
	\end{equation}
	The limit in \eqref{eq:A-n-gamma-limit} is a strictly positive
	deterministic constant. In particular,
	\[
	\mathbb P\bigl(A_{n,\gamma}(x)>0\bigr)
	\longrightarrow1.
	\]
	
	On the event \(\{A_{n,\gamma}(x)>0\}\), we have
	\begin{align*}
		\frac{\log A_{n,\gamma}(x)}{\log n}
		=
		1-
		\frac{\gamma}{2}
		\left(1-\frac{\beta}{\alpha}\right)
		+
		\frac{
			\log\left[
			n^{
				\frac{\gamma}{2}
				\left(1-\frac{\beta}{\alpha}\right)-1
			}
			A_{n,\gamma}(x)
			\right]
		}{
			\log n
		}.
	\end{align*}
	By \eqref{eq:A-n-gamma-limit} and the continuous mapping theorem,
	the numerator in the second term is bounded in probability.
	Since \(\log n\to\infty\), it follows that
	\[
	\frac{\log A_{n,\gamma}(x)}{\log n}
	\xrightarrow{\mathbb P}
	1-
	\frac{\gamma}{2}
	\left(1-\frac{\beta}{\alpha}\right).
	\]
	Therefore,
	\begin{align*}
		\widehat r_{n,\gamma}
		\xrightarrow{\mathbb P}
		1-\frac2\gamma
		+
		\frac2\gamma
		\left[
		1-
		\frac{\gamma}{2}
		\left(1-\frac{\beta}{\alpha}\right)
		\right]
		=
		\frac{\beta}{\alpha}.
	\end{align*}
\end{proof}

\begin{remark}
	\label{rem:individual-alpha-beta-estimators}
	The quadratic case \(\gamma=2\) is particularly convenient, since
	\begin{equation}
		\label{eq:beta-alpha-ratio-estimator-quadratic}
		\widehat r_{n,2}
		=
		\frac{\log A_{n,2}(x)}{\log n}
		\xrightarrow{\mathbb P}
		\frac{\beta}{\alpha}.
	\end{equation}
	This estimator does not involve \(\nu\) or
	\(C_{\alpha,\beta}\).
	
	If \(\alpha\) is known, then
	\[
	\widehat\beta_n
	:=
	\alpha\widehat r_{n,2}
	\]
	is a consistent estimator of \(\beta\). If \(\beta\) is known, then,
	on the event \(\{\widehat r_{n,2}>0\}\), the estimator
	\[
	\widehat\alpha_n
	:=
	\frac{\beta}{\widehat r_{n,2}}
	\]
	satisfies
	\[
	\widehat\alpha_n
	\xrightarrow{\mathbb P}
	\alpha.
	\]
	An arbitrary value may be assigned on
	\(\{\widehat r_{n,2}\leq0\}\), whose probability tends to zero.
	
	If both \(\alpha\) and \(\beta\) are unknown, temporal
high-frequency observations identify only their ratio and do not
separate the two parameters. If, in addition, a consistent estimator
\(\widetilde\alpha_n\) of \(\alpha\) is available, for instance one
constructed from spatial high-frequency observations, then
\[
\widetilde\beta_n
:=
\widetilde\alpha_n\widehat r_{n,2}
\]
satisfies, by Slutsky's theorem,
\[
\widetilde\beta_n
\xrightarrow{\mathbb P}
\beta.
\]
\end{remark}

\end{document}